\documentclass[a4paper,12pt]{article}
\usepackage[utf8]{inputenc}
\usepackage[english]{babel}
\usepackage{enumitem}
\usepackage{color}
\usepackage{amsmath}
\usepackage{amssymb}
\usepackage{amsthm}
\usepackage{amsfonts}
\usepackage{epsfig}
\usepackage{graphicx}
\usepackage{epstopdf}
\usepackage{multirow}
\usepackage{cite}
\usepackage{caption}
\usepackage{subfigure}
\newtheorem{theorem}{Theorem}[section]
\newtheorem{definition}{Definition}[section]
\newtheorem{lemma}{Lemma}[section]
\newtheorem{proposition}{Proposition}[section]

\newtheorem{remark}{Remark}[section]
\newtheorem{assumption}{Assumption}[section]
\newtheorem{algorithm}{Algorithm}[section]

\theoremstyle{example}
\theoremstyle{definition}
\theoremstyle{Assumption}

\DeclareMathOperator*{\gra}{gra}
\DeclareMathOperator*{\dom}{dom}
\DeclareMathOperator*{\zer}{zer}
\DeclareMathOperator*{\Id}{Id}

\DeclareMathOperator*{\nt}{int}
\allowdisplaybreaks

\begin{document}
	\title{\textbf{Nonlinear three-operator splitting algorithms with momentum for monotone inclusions}}
	\author{ \sc \normalsize Liqian Qin$^{a}${\thanks{ email: qlqmath@163.com}},\,\,Aviv Gibali$^{b}${\thanks{ email: avivgi@hit.ac.il}},\,\, Cuijie Zhang$^{c}${\thanks{ email: cuijie\_zhang@126.com}},\,\,
   Yuchao Tang$^{a}${\thanks{ Corresponding author. email: hhaaoo1331@163.com}},\,\, \\
		\small $^a$School of Mathematics and Information Science, Guangzhou University,\\
 \small  Guangzhou 510006, P.R. China \\
 \small $^{b}$Department of Applied Mathematics, Holon Institute of Technology,\\
\small  Holon 5810201, Israel\\
\small $^{c}$ College of Science, Civil Aviation University of China, \\
\small Tianjin, 300300, P.R. China \\
	}
	\date{}
	\maketitle
	
	\begin{abstract}

		{In this paper, we introduce three novel splitting algorithms for solving structured monotone inclusion problems involving the sum of a maximally monotone operator, a monotone and Lipschitz continuous operator and a cocoercive operator. Each proposed method extends one of the classical schemes: the semi-forward-reflected-backward splitting algorithm, the semi-reflected-forward-backward splitting algorithm, and the outer reflected forward-backward splitting algorithm by incorporating a nonlinear momentum term. Under appropriate step-size conditions, we establish the weak convergence of all three algorithms, and further prove their  $R$-linear convergence rates under strong monotonicity assumptions. Preliminary numerical experiments on both synthetic datasets and real-world quadratic programming problems in portfolio optimization demonstrate the effectiveness and superiority of the proposed algorithms.
		}
	\end{abstract}
	
	\noindent{\bf Keywords:} Monotone inclusion problem; Semi-forward-reflected-backward splitting; Semi-reflected-forward-backward splitting; Outer reflected forward-backward splitting; Momentum.

    \noindent {\bf AMS Subject Classification}:  47H05, 65K15, 90C25

	\section{ Introduction }
Monotone inclusion problems form a cornerstone of modern optimization theory and arise in a wide range of applications, including signal and image processing \cite{Combettes2005, Combettes2012, Condat2023}, machine learning \cite{Boyd2011,Barnert,Bot2023}, image reconstruction \cite{Bot2014, BricenoArias2011}, and game theory \cite{BricenoArias2019}, etc. A key challenge in this area is the resolution of composite inclusions that involve the sum of several monotone operators, where the specific structure of each operator plays a crucial role in the development of efficient numerical methods. In this work, we focus on the following monotone inclusion problem:
	
\noindent{\bf Problem.} Consider a real Hilbert space $\mathcal{H}$. Let $A: \mathcal{H} \rightarrow 2^{\mathcal{H}}$ be a maximally monotone operator, let $B: \mathcal{H} \rightarrow \mathcal{H}$ be a monotone and $\mu$-Lipschitz operator, and let $C: \mathcal{H} \rightarrow \mathcal{H}$ be $\frac{1}{\beta}$-cocoercive operator for some $\beta \in (0, +\infty).$ The problem is to find $x \in \mathcal{H}$  such that
	\begin{equation}
		\label{ABC}
		0 \in Ax+Bx+Cx,
	\end{equation}
where the solution set, denoted by $\zer(A+B+C)$, is assumed to be nonempty.

Solving a monotone inclusion problem involving three operators can be viewed as a natural extension of splitting algorithms for the two-operator case. In the classical two-operator setting (with operators $A$ and $B$), Tseng’s algorithm-a forward-backward-forward splitting method  \cite{Tseng}– finds a zero of $A + B$ by performing one backward (implicit) step for $A$ and two forward evaluations of the single-valued Lipschitz monotone operator $B$ per iteration. A naive attempt to handle three operators (say $A$, $B$, and $C$) is to combine $B + C$ into one Lipschitz monotone operator and then apply Tseng’s method \cite{Tseng}. In theory, this works because the sum of a cocoercive operator and a Lipschitz monotone operator is still monotone and Lipschitz. In practice, however, this direct approach is inefficient, each iteration would require evaluating both $B$ and $C$ twice, failing to exploit the special cocoercivity property of $C$. Clearly, a more tailored approach is needed to handle three-operator problems efficiently. Brice\~{n}o-Arias and Davis \cite{FBHF} addressed the above challenge by leveraging the cocoercivity of operator $C$ to design the forward-backward-half forward splitting (FBHFS) algorithm as follows:
\begin{equation}
		\label{FBHF}
			\left\{
		\begin{array}{lr}
			y_k=J_{\gamma_k A}(x_k-\gamma_k(B+C)x_k), & \\
			x_{k+1}=y_k+\gamma_k(Bx_k-By_k). & \\
		\end{array}
		\right.
\end{equation}
where $ \gamma_k $ is the step-size,  $J_{\gamma_k A}=( \Id+\gamma_k A)^{-1}$ is the resolvent of $A$. They proved the weak convergence of the FBHFS algorithm \eqref{FBHF} with a variable step-size, $\gamma_{k}\in [\eta , \chi-\eta]$, $\eta \in (0,\frac{\chi}{2})$, $\chi=\frac{4 }{\beta+\sqrt{\beta^2+16\mu^2}}$, as well as with a line search procedure. This method modifies Tseng’s scheme \cite{Tseng} to avoid redundant computations on the cocoercive part. Each iteration of the FBHFS algorithm \eqref{FBHF} still performs two forward evaluations of $B$ (the Lipschitz monotone operator), as in Tseng’s method  \cite{Tseng}, but requires only one forward evaluation of $C$, the cocoercive operator. By halving the forward steps on $C$, the algorithm significantly reduces per-iteration computational cost compared to the naive approach. Importantly, the FBHFS algorithm \eqref{FBHF} seamlessly recovers previous splitting algorithms as special cases: if $C=0$ (no cocoercive part), it reduces exactly to Tseng’s forward-backward-forward algorithm  \cite{Tseng}, and if $B=0$ (no Lipschitz part), it becomes the classic forward-backward splitting method \cite{Lions,Passty}. Furthermore, the authors also considered the FBHFS algorithm \eqref{FBHF} employing non-self-adjoint linear operators, which is defined by
\begin{equation}
	\label{nonFBHF}
	\left\{
	\begin{array}{lr}
		y_k=(M_k+A)^{-1}(M_kx_k-(B+C)x_k), & \\
		x_{k+1}=x_k+\lambda_k(M_k(y_k-x_k)+Bx_k-By_k). & \\
	\end{array}
	\right.
\end{equation}
where $M_k:\mathcal{H} \to \mathcal{H}$ are bounded linear operators, $\forall k \in \mathbb{N}$. Under suitable conditions on $M_k$ and $\lambda_k$, the weak convergence of the iterative scheme \eqref{nonFBHF} is established.

Subsequently, Malitsky and Tam \cite{Malitsky} introduced the semi-forward-reflected-back-ward splitting (SFRBS) algorithm, defined by
\begin{equation}
		\label{SFRBS}
			x_{k+1}=J_{\gamma A}(x_{k}-2\gamma B x_k+\gamma Bx_{k-1}-\gamma Cx_k),
\end{equation}
and the convergence of $\{x_k\}$ is obtained by the step size  satisfies $\gamma \in (0, \frac{2}{4\mu+\beta}).$ In contrast, to address problem \eqref{ABC}, Cevher and Vu \cite{Cevher2019} proposed the semi-reflected-forward-backward splitting (SRFBS) algorithm, which is defined as
\begin{equation}
	\label{SRFBS}
x_{k+1} = J_{\gamma A}(x_k-\gamma B(2x_k-x_{k-1})-\gamma Cx_k).
\end{equation}
The convergence of the SRFBS algorithm \eqref{SRFBS} was analyzed in \cite{Cevher2019} under suitable assumptions on the step size $\gamma$. More recently, Yu et al. \cite{Yu} introduced the outer reflected forward-backward splitting (ORFBS) algorithm, which iterates as follows:
\begin{equation}
\label{ORFBS}
x_{k + 1} = J_{\gamma A}(x_k - \gamma B x_k - \gamma C x_k) - \gamma (B x_k - B x_{k - 1}).
\end{equation}
Under the step size condition $\gamma < \min\left\{ \left(\frac{2}{\beta}-\epsilon_2\right)\epsilon_1, \ (3-\epsilon_3)\epsilon_2, \ \frac{1/2-\epsilon_1-1/\epsilon_3}{\mu}\right\}$ with $\epsilon_1, \epsilon_2, \epsilon_3 > 0$, they proved the weak convergence of the sequence $\{x_k\}_{k \in \mathbb{N}}$ generated by the ORFBS algorithm \eqref{ORFBS}. Motivated by the success of inertial and accelerated methods, several authors have proposed accelerated variants of the algorithms FBHFS \eqref{FBHF}, SFRBS \eqref{SFRBS}, SRFBS \eqref{SRFBS} and ORFBS \eqref{ORFBS}; see, for example, \cite{Izuchukwu2025,Zong2024a,Dong2024,Yao2024,Shehu2024,Zong2022,Qin2024,Zong2023,Tan2022,Izuchukwu2023,Tan2025,Yao2025,Zhang2025,Tongnoi2025} and references therein.

The resolvent operator $J_{\gamma A} = (\Id + \gamma A)^{-1}$ serves as the fundamental backward step in a broad class of operator splitting algorithms, whose analysis and design  heavily rely on this operator. To enhance the applicability of these methods, this study explores an innovative approach centered around a nonlinear resolvent of the form $(M + A)^{-1} M$. Here, $M: \mathcal{H} \to \mathcal{H}$ acts as a nonlinear kernel,  also referred to as a warped resolvent \cite{Giselsson,bui}. This generalization provides a powerful modeling tool for a wide spectrum of monotone inclusion algorithms, offering a unified framework that encompasses both existing and novel schemes. The central idea is to
exploit a problem adapted kernel $M$ that captures the intrinsic structure of the underlying problem. In this context, numerous studies have proposed preconditioned variants of classic algorithms, which enable flexible and efficient problem dependent preconditioning \cite{Bredies2015,Raguet2015,Bredies2017}. In particular, the work of \cite{Bredies} introduced the concept of an admissible preconditioner. Specifically, for an operator $A: \mathcal{H} \rightarrow 2^{\mathcal{H}}$,  a bounded linear operator $M$ is said to be an admissible preconditioner if it is self-adjoint, positive semidefinite, and satisfies that  $(M+A)^{-1}M$ is single-valued with full domain. Building upon this notion, the authors proposed the degenerate preconditioned proximal point (Degenerate PPP) method \cite{Bredies} and established its convergence. As stated in \cite{Bredies2015}, the Douglas–Rachford splitting algorithm \cite{Douglas1956} can be viewed as a particular instance of the degenerate preconditioned proximal point framework. Furthermore, Sun et al. \cite{Sun2025} developed an accelerated preconditioned alternating direction method of multiplier (pADMM) by combining the degenerate PPP method \cite{Bredies}  with  the fast  Krasnosel'ski\u {\i}-Mann iteration \cite{fastKM}. Based on the Halpern iteration \cite{Halpern1967}, Zhang et al. subsequently proposed the Halpern-type preconditioned proximal point algorithm \cite{Zhang2025a}.

While the convergence-guaranteeing methods in \cite{bui, Giselsson} rely on corrective projection steps, Morin et al. \cite{Morin} proposed an alternative strategy that incorporates a nonlinear momentum term, which typically resulting in a lower per-iteration computational cost. Specifically, they introduced a nonlinear forward-backward splitting algorithm with momentum correction to solve the two-operator monotone inclusion problem involving a maximally monotone operator and a cocoercive operator. The algorithm is stated as follows:
\begin{equation}
\label{maincite}
\aligned
&x_{k+1}=(M_k+A)^{-1}(M_kx_k-Cx_k+\gamma_k^{-1}u_k), \\
&u_{k+1}=(\gamma_kM_k-S)x_{k+1}-(\gamma_kM_k-S)x_k.
\endaligned
\end{equation}
Here, $\gamma_k > 0$, $M_k: \mathcal{H} \rightarrow \mathcal{H}$ are general nonlinear operators, and $S$ is a bounded linear, self-adjoint, strongly positive operator. For the convergence analysis, it is assumed that $\gamma_k M_k - S$ is $L_k$-Lipschitz continuous with respect to the norm induced by $S$ for some $L_k \geq 0$, for all $k \in \mathbb{N}$. The generality of algorithm \eqref{maincite} has been demonstrated by its ability to recover several known methods, such as the forward-backward splitting algorithm \cite{Lions, Passty} and the forward-reflected-backward method \cite{Malitsky}, etc. Very recently,  Rold\'{a}n and Vega \cite{Roldan2025} proposed an inertial and relaxation extension of \eqref{maincite}. In contrast, in \cite{TangQin}, we introduced a FBHFS algorithm with momentum.

This paper aims to further extend the momentum technique to the splitting algorithms, namely the SFRBS algorithm  \eqref{SFRBS}, the SRFBS algorithm \eqref{SRFBS}, and the ORFBS algorithm \eqref{ORFBS}, respectively. To this end, we propose three novel algorithms that integrate warped resolvents with momentum for solving the monotone inclusion problem \eqref{ABC}. The main contributions of this paper are summarized as follows:
\begin{itemize}
	\item[{\rm(i)}] We propose three nonlinear splitting algorithms with momentum for solving the monotone inclusion problem \eqref{ABC}: the nonlinear semi-forward-reflected-backward splitting algorithm with momentum, the nonlinear semi-reflected-forward-backward splitting algorithm with momentum, and the nonlinear outer reflected forward-backward splitting algorithm with momentum. Weak convergence of all three algorithms is established under suitable parameter conditions. Furthermore, under strong monotonicity assumptions, $R$-linear convergence rates are proved for Algorithms \ref{algorithm2}, \ref{algorithm3}, and \ref{algorithm4}, respectively.
	\item[{\rm(ii)}] The proposed framework demonstrates strong generality and flexibility:  several existing algorithms can be recovered as special cases.  Beyond unifying known methods, the introduction momentum terms enables the design of new splitting schemes for monotone inclusions involving sums of four operators.

	\item[{\rm(iii)}] The performance of the proposed algorithm is comprehensively validated through applications to two representative problems: constrained nonlinear optimization and portfolio optimization. Numerical experiments confirm their efficiency and robustness compared with existing approaches.
\end{itemize}

This paper is organized as follows. Section 2 reviews essential fundamental definitions and lemmas required for the subsequent analysis. In Section 3, we introduce three novel splitting algorithms with momentum for solving the problem \eqref{ABC}: the nonlinear semi-forward-reflected-backward splitting algorithm with momentum, the nonlinear semi-reflected-forward-backward splitting algorithm with momentum, and the nonlinear outer reflected forward-backward splitting algorithm with momentum. We establish their weak convergence under standard assumptions and prove linear convergence under additional conditions of strong monotonicity. Section 5 presents numerical experiments on synthetic datasets and real-world quadratic programming problems from portfolio optimization.  Finally, we give some conclusions.

	\section{ Preliminaries}
	{
In this paper,  $\mathcal{H}$ and $\mathcal{G}$ denote real Hilbert spaces endowed with the inner product  $\langle \cdot, \cdot\rangle$ and let associated  norm $\|\cdot\|$. Let $\mathbb{N}$ be the set of nonnegative integers and $\mathbb{N}_0$ the set of positive integers. The symbols $``\rightarrow"$ and $``\rightharpoonup"$ represent strong convergence and weak convergence, respectively. The identity operator is denoted by $\Id$.  Let $L: \mathcal{H} \rightarrow \mathcal{G}$ be a bounded linear operator with adjoint $L^*: \mathcal{G} \rightarrow \mathcal{H}$ and operator norm $\|L\|$. An operator $S$ is said to be self-adjoint if $S^* = S$, and strongly monotone if there exists a constant $c \in (0, +\infty)$ such that
\begin{equation*}
\langle Sx | x \rangle \geq c \|x\|^2, \quad \forall x \in \mathcal{H}.
\end{equation*}
Define  the set of strongly monotone self-adjoint operators on $\mathcal{H}$ as
$$\mathcal{P}(\mathcal{H})=\{S:\mathcal{H} \rightarrow \mathcal{H}|S \ \hbox{is linear, self-adjoint and strongly monotone}\}.$$
If $S \in \mathcal{P}(\mathcal{H})$, then $S$ is invertible and its inverse $S^{-1}$ also belongs to $\mathcal{P}(\mathcal{H}).$  The inner product induced by $S$ is defined as $\langle x , y  \rangle_S := \langle S x,y \rangle$, and the corresponding norm is denoted by $\|\cdot\|_S$.  Note that $S =S^{\frac{1}{2}}\circ S^{\frac{1}{2}},$ where $S^{\frac{1}{2}}$ is a self-adjoint, strongly monotone linear operator. Consequently, the Cauchy–Schwarz inequality extends naturally to the norms induced by $S$ and $S^{-1},$ taking the following form:
\begin{equation*}
|\langle x,u \rangle| = |\langle S^{-\frac{1}{2}}x,S^{\frac{1}{2}}u \rangle|\leq  \| S^{-\frac{1}{2}}x\| \|S^{\frac{1}{2}}u \| =
\|x\|_{S^{-1}}\|u\|_S, \quad \forall (x,u) \in \mathcal{H}\times\mathcal{H}.
\end{equation*}
Moreover, the following identity holds:
\begin{equation*}
2\langle x - y | y - z\rangle_S = \|x - z\|_S^2 - \|x - y\|_S^2 - \|y - z\|_S^2.
\end{equation*}

\begin{definition}
	\rm
	Let $A : \mathcal{H} \rightarrow 2^\mathcal{H}$ be a set-valued operator, we have the following notations:
	\begin{itemize}
		\item[{\rm(i)}] the  graph of $A$ is defined by $\gra (A) = \{(x, y) \in  \mathcal{H}^2 \ | \ y \in Ax\}$;
		
		\item[{\rm(ii)}] the set of zeros of $A$ is defined by $\zer (A) = \{x \in \mathcal{H}\ | \ 0 \in Ax\} =A^{-1}(0);$
		
		\item[{\rm(iii)}] the domain of $A$ is defined by $\dom (A) = \{x \in \mathcal{H}\ | Ax \neq \emptyset \};$
		\item[{\rm(iv)}] the  resolvent of $A$ with $\gamma >0$ is defined by
		$$
		J_{\gamma A} = (\Id + \gamma A)^{-1}.
		$$
	\end{itemize}
	Let $x,y\in\mathcal{H}$. Then
	\begin{equation*}
		\label{Ros}
		y\in J_{\gamma A}x\Longleftrightarrow x-y\in \gamma Ay.
	\end{equation*}
\end{definition}

\begin{definition}{\rm(\cite{BC2011}\ )}
			{\rm
			Let $A: \mathcal{H} \rightarrow 2^{\mathcal{H}}$ be	a set-valued mapping,
				 $A: \mathcal{H} \rightarrow 2^{\mathcal{H}}$ is said to be
				\begin{itemize}
					\item[(i)]  monotone if $\langle u-v,x-y \rangle \geq 0$ for all $(x,u),(y,v) \in \gra(A)$.
					\item[(ii)]  maximally monotone if there exists no monotone operator $B: \mathcal{H} \rightarrow 2^{\mathcal{H}}$ such that gra$(B)$ properly contains gra$(A),$ i.e., for every $(x,u) \in \mathcal{H}\times\mathcal{H},$
					$$
					(x,u) \in \gra (A)  \ \ \Leftrightarrow \ \ \langle u-v , x-y \rangle \geq 0,  \ \  \forall(y,v)\in \gra (A).
					$$
				\end{itemize}
			}
		\end{definition}
		\begin{definition}
			{\rm
				A single-valued operator  $T: \mathcal{H} \rightarrow {\mathcal{H}}$ is said to be
				\begin{itemize}
					\item[(i)]  $L$-Lipschitz continuous w.r.t. $S$ if there exists a constant $ L > 0$, such that
					$$
					\|Tx-Ty\|_{S^{-1}}\leq  L\|x-y\|_S,\quad\forall x, y\in \mathcal{H}.
					$$
					\noindent
					\item[(ii)]  $\beta^{-1}$-cocoercive operator w.r.t. $S$ if there exists a constant $ \beta > 0$, such that
					$$
					\langle Tx-Ty,x-y \rangle \geq  \beta^{-1}\|Tx-Ty\|_{S^{-1}}^2,\quad\forall x, y\in \mathcal{H}.
					$$
				\end{itemize}
			}
		\end{definition}
\begin{lemma}{\rm(\cite{{Morin}}\ )}\label{lem0}
Let $C$ be the $\beta^{-1}$-cocoercive operator w.r.t. $S$, for some $\beta>0$. Then the following inequality holds:
\begin{equation*}
\langle Cx - Cy, z - y\rangle\geq-\frac{\beta}{4}\|z-x\|_{S}^{2},\quad\forall x,y,z\in \mathcal{H}.
\end{equation*}
\end{lemma}
		
\begin{lemma}{\rm(\cite{BC2011}\ )}\label{lem3}
Let $A: \mathcal{H} \rightarrow 2^{\mathcal{H}}$ be maximally monotone. Then $\gra(A)$ is sequentially closed in $\mathcal{H}^{weak}\times\mathcal{H}^{strong}$, i.e., for every sequence $(x_{k},y_{k})_{k \in \mathbb{N}}$ in $\gra(A)$ and $(x,y)\in \mathcal{H}\times\mathcal{H}$, if $x_{k}\rightharpoonup x$ and $y_{k}\rightarrow y$, then $(x,y) \in \gra(A)$. \vskip 1mm
\end{lemma}
\begin{lemma}\label{lem1}{\rm(\cite{BC2011}\ )}
Let $D$ be a nonempty set of $\mathcal{H}$, and $\{x_{k}\}_{k \in \mathbb{N}}$ be a sequence in $\mathcal{H}$.  If the following conditions hold:
\begin{itemize}
\item[{\rm(i)}]  For every $x \in D$, $\lim\limits_{k\rightarrow \infty}\|x_{k}-x\|$ exists;
\noindent
\item[{\rm(ii)}] Every weak sequential cluster point of $\{x_{k}\}_{k \in \mathbb{N}}$ belongs to $D$.
\end{itemize}
Then the sequence $\{x_{k}\}_{n \in \mathbb{N}}$ converges weakly to a point in $D$.
\end{lemma}

\section{Main algorithms and convergence analysis}
In this section, we propose three novel splitting algorithms for solving the monotone inclusion problem \eqref{ABC}: the nonlinear semi-forward-reflected-backward splitting (SFRBS) algorithm, the nonlinear semi-reflected-forward-backward splitting (SRFBS) algorithm, and the nonlinear outer reflected forward-backward splitting (ORFBS) algorithm. The key assumptions required for their convergence analyzes are presented below.
			\begin{assumption}
				\label{assumption1}
				
				Assume that
				\begin{itemize}
					\item[{\rm(i)}]  $A:\mathcal{H} \rightarrow 2^{\mathcal{H}}$ is  maximally monotone.
					
					\item[{\rm(ii)}]   $B:\mathcal{H} \rightarrow \mathcal{H}$ is single-valued monotone and $\mu$-Lipschitz w.r.t. $S$, where $S \in \mathcal{P}(\mathcal{H}).$

					\item[{\rm(iii)}]  $C:\mathcal{H}\rightarrow \mathcal{H}$ is $\beta^{-1}$-cocoercive w.r.t. $S$, where $S \in \mathcal{P}(\mathcal{H}),$ for some $\beta>0$.
					
					\item[{\rm(iv)}]   The solution set of \eqref{ABC}, denoted by,
					$$
					\zer(A+B+C):=\{x \in \mathcal{H}:0 \in Ax+Bx+Cx\}
					$$
					is nonempty.
                  \item[{\rm(iv)}] The nonlinear kernel $M_k : \mathcal{H} \rightarrow \mathcal{H}$ be such that  $\gamma_k M_k - S$ is $L_k$-Lipschitz continuous w.r.t. $S$, for some $L_k \in [0,1)$, $\gamma_k\geq \gamma$ for some $\gamma>0$, $\forall k \in \mathbb{N}.$
				\end{itemize}
			\end{assumption}
\begin{proposition}{\rm(\cite{Morin}\ )}
Let Assumption \ref{assumption1} (iv) hold. Then, for all $k \in \mathbb{N}$, the nonlinear operators $M_k$ are $2\gamma^{-1}$-Lipschitz continuous w.r.t. $S$, maximally monotone, and strongly monotone w.r.t. $S$.
\end{proposition}		

\subsection{Nonlinear semi-forward-reflected-backward splitting algorithm with momentum}
In this subsection, we propose a nonlinear semi-forward-reflected-backward splitting (SFRBS) algorithm with momentum, denoted as Algorithm \ref{algorithm2},  and establish its convergence results.
\begin{algorithm}\label{algorithm2}
\hrule
\noindent\textbf{\footnotesize{Nonlinear SFRBS algorithm with momentum}}
\hrule
\vskip 1mm
\noindent Let $S \in \mathcal{P}(\mathcal{H})$, $M_k: \mathcal{H} \rightarrow \mathcal{H}$,  $\gamma_k > 0$ $(\forall k \in\mathbb{N})$ and initial point $x_{0}, \ u_{0} \in \mathcal{H}.$ For $k=0,1,\cdots,$ do \\
\begin{equation*}
\label{algorithm_2}
\aligned
&x_{k+1}=(M_k+A)^{-1}(M_kx_k-2Bx_k+Bx_{k-1}-Cx_k+\gamma_k^{-1}u_k), \\
&u_{k+1}=(\gamma_k M_k-S)x_{k+1}-(\gamma_k M_k-S)x_k.
\endaligned
\end{equation*}
				\vskip 1mm
				
				\hrule
				
				\hspace*{\fill}
			\end{algorithm}
For the iterates $\{x_k\}_{k \in \mathbb{N}}$ and $\{u_k\}_{k \in \mathbb{N}}$ generated by Algorithm \ref{algorithm2} and $ \forall x \in \zer(A+B+C),$ we define the following function,
\begin{equation}\label{Psik}
\begin{aligned}
\Psi_{k}(x):=&\|x_k-x\|_{S}^2+2\langle u_k, x_k-x\rangle +(\gamma_k\mu+L_{k-1})\|x_k-x_{k-1}\|_S^2 \\
&+2\gamma_k\langle Bx_k-Bx_{k-1},x-x_k\rangle,
\end{aligned}
\end{equation}
which helps to establish the weak convergence of the Algorithm \ref{algorithm2}.		
\begin{lemma}
\label{lemma2}
Suppose that Assumption {\rm\ref{assumption1}} holds. Then for  $\{x_k\}_{k \in \mathbb{N}}$ and $\{u_k\}_{k \in \mathbb{N}}$ generated by Algorithm \ref{algorithm2} and  $\forall x^* \in \zer(A+B+C),$
\begin{equation}
\Psi_{k+1}(x^*) \leq \Psi_{k}(x^*)-(1- L_{k-1}-L_k-\gamma_k\mu-\gamma_{k+1}\mu-\frac{\gamma_k\beta}{2})\|x_{k+1}-x_k\|_S^2 \label{17}
\end{equation}
holds for all $n \in \mathbb{N}_0.$
\end{lemma}
\begin{proof}
From Algorithm \ref{algorithm2} we know that
$$x_{k+1}=(M_k+A)^{-1}(M_kx_k-2Bx_k+Bx_{k-1}-Cx_k+\gamma_k^{-1}u_k),$$
which is equivalent to the inclusion
\begin{equation}
M_kx_k-M_kx_{k+1}-2Bx_k+Bx_{k-1}-Cx_k+\gamma_k^{-1}u_k \in Ax_{k+1}.     \label{18}
\end{equation}
Since $x^* \in \zer(A+B+C),$ we have
\begin{equation}
-(B+C)x^* \in Ax^*.                                 \label{19}
\end{equation}
Combining \eqref{18}, \eqref{19} and the monotonicity of $A$, we obtain
\begin{equation}
0 \leq \langle M_kx_k-M_kx_{k+1}-2Bx_k+Bx_{k-1}-Cx_k+\gamma_k^{-1}u_k+(B+C)x^* , x_{k+1}-x^* \rangle. 					\label{20}
\end{equation}
Multiplying both sides by $2\gamma_k$ and substituting the definition of $u_{k+1}$ yields
\begin{equation}
\aligned
0& \leq 2\langle Sx_k-Sx_{k+1}+u_k-u_{k+1}-2\gamma_k Bx_k+\gamma_k Bx_{k-1}-\gamma_k Cx_k+\gamma_k(B+C)x^* , x_{k+1}-x^* \rangle\\
& = 2\langle Sx_k-Sx_{k+1}, x_{k+1}-x^* \rangle +2\langle u_k-u_{k+1}, x_{k+1}-x^* \rangle-2\gamma_k \langle Bx_k-Bx^*, x_{k+1}-x^*\rangle \\
& \quad +2\gamma_k \langle Bx_k-Bx_{k-1},x^*-x_k\rangle+2\gamma_k \langle Bx_k-Bx_{k-1},x_k-x_{k+1}\rangle \\
& \quad -2\gamma_k\langle Cx_k-Cx^*,x_{k+1}-x^*\rangle. 	
\endaligned
\label{21}
\end{equation}
In the following, we analyze each term on the right-hand side of the inequality one by one. Firstly, using the identity $2\langle x - y | y - z\rangle_S = \|x - z\|_S^2 - \|x - y\|_S^2 - \|y - z\|_S^2$, we have
				\begin{equation}
					\begin{aligned}
					   & 2\langle Sx_k -Sx_{k+1}, x_{k+1}-x^*\rangle \\
                       =&2\langle x_k-x_{k+1},x_{k+1}-x^*\rangle_S \\
                       =&\|x_k-x^*\|_S^2-\|x_{k+1}-x^*\|_S^2-\|x_{k+1}-x_k\|_S^2.
					\end{aligned}  \label{22}
				\end{equation}
Secondly, by Assumption \ref{assumption1} and the definition of $u_k$, for all $k \in \mathbb{N}$,
				\begin{equation}
					\begin{aligned}
						& 2\langle u_k-u_{k+1}, x_{k+1}-x^*\rangle \\
                        =&2\langle u_k,x_k-x^*\rangle+2\langle u_k,x_{k+1}-x_k\rangle-2\langle u_{k+1},x_{k+1}-x^*\rangle  \\
                        \leq & 2\langle u_k,x_k-x^*\rangle-2\langle u_{k+1},x_{k+1}-x^*\rangle+L_{k-1}\|x_k-x_{k-1}\|_S^2 \\
                        &+L_{k-1}\|x_{k+1}-x_k\|_S^2.
					\end{aligned}  \label{23}
				\end{equation}
Thirdly, since $B:\mathcal{H} \rightarrow \mathcal{H}$ is monotone and $\mu$-Lipschitz w.r.t. $S$, it follows that
\begin{equation}
\begin{aligned}
&\quad -2\gamma_k \langle Bx_k-Bx^*, x_{k+1}-x^*\rangle+2\gamma_k \langle Bx_k-Bx_{k-1},x^*-x_k\rangle\\
&\quad +2\gamma_k \langle Bx_k-Bx_{k-1},x_k-x_{k+1}\rangle \\
&=2\gamma_k \langle Bx_{k+1}-Bx^*, x^*-x_{k+1}\rangle+ 2\gamma_k \langle Bx_k-Bx_{k+1}, x^*-x_{k+1}\rangle\\
& \quad + 2\gamma_k \langle Bx_k-Bx_{k-1},x^*-x_k\rangle+2\gamma_k \langle Bx_k-Bx_{k-1},x_k-x_{k+1}\rangle \\
& \leq  2\gamma_k \langle Bx_k-Bx_{k+1}, x^*-x_{k+1}\rangle+ 2\gamma_k \langle Bx_k-Bx_{k-1},x^*-x_k\rangle \\
& \quad +\gamma_k\mu \|x_k-x_{k-1}\|_S^2+\gamma_k\mu\|x_{k+1}-x_k\|_S^2.
\end{aligned}  \label{24}
\end{equation}
By Lemma \ref{lem0}, we have
\begin{equation}
-2\gamma_k\langle Cx_k-Cx^* , x_{k+1}-x^* \rangle \leq \frac{\gamma_k\beta}{2}\|x_{k+1}-x_k\|_S^2. \label{25}
\end{equation}
From \eqref{21}--\eqref{25}, it follows that
				\begin{equation*}
					\begin{aligned}
						& \quad\|x_{k+1}-x^*\|_S^2+2\langle u_{k+1}, x_{k+1}-x^*\rangle+2\gamma_k\langle Bx_{k+1}-Bx_k, x^*-x_{k+1}\rangle\\
						& \leq \|x_k-x^*\|_S^2 +2\langle u_k,x_k-x^*\rangle+(\gamma_k\mu+L_{k-1})\|x_k-x_{k-1}\|_S^2 \\
                        & \quad +2\gamma_k \langle Bx_k-Bx_{k-1},x^*-x_k\rangle -(1-L_{k-1}-\gamma_k\mu-\frac{\gamma_k\beta}{2})\|x_{k+1}-x_k\|_S^2.
					\end{aligned} 					\label{26}
				\end{equation*}		
Then, \eqref{17} follows from the definition of $\Psi_k(x)$ in \eqref{Psik}.
			\end{proof}

\begin{theorem}
\label{theorem2}
{\noindent
Suppose that Assumption {\rm\ref{assumption1}} holds and there exists an $\epsilon >0$ such that
\begin{equation}
1- L_{k-1}-L_k-\gamma_k\mu-\gamma_{k+1}\mu-\frac{\gamma_k\beta}{2}\geq \epsilon,		\label{27}
\end{equation}
holds for all $n \in \mathbb{N}_0$. Then the sequence $\{x_k\}_{k \in \mathbb{N}}$ generated by Algorithm \ref{algorithm2} converges weakly to a point $x^*$ in $\zer(A+B+C)$.
				}
\end{theorem}

\begin{proof}
By combining the definition of $\Psi_k(x)$ in \eqref{Psik}, the Lipschitz property of the operator $\gamma_kM_k-S, \ \forall k \in \mathbb{N}_0,$ and the Cauchy-Schwarz inequality, we deduce that
\begin{equation}
\begin{aligned}
\Psi_{k}(x^*)&=\|x_k-x^*\|_{S}^2+2\langle u_k, x_k-x^*\rangle +(\gamma_k\mu+L_{k-1})\|x_k-x_{k-1}\|_S^2\\
& \quad +2\gamma_k\langle Bx_k-Bx_{k-1},x-x_k\rangle  \\
& \geq \|x_k-x^*\|_{S}^2-L_{k-1}(\|x_k-x_{k-1}\|_S^2+\|x_k-x\|_S^2)+(\gamma_k\mu+L_{k-1})\|x_k-x_{k-1}\|_S^2 \\
& \quad -\gamma_k\mu\|x_k-x_{k-1}\|_S^2-\gamma_k\mu\|x_k-x\|_S^2\\
& = (1-L_{k-1}-\gamma_k\mu)\|x_k-x^*\|_S^2> 0.
\end{aligned}  \label{28}
\end{equation}
Hence, the sequence $\{\Psi_{k}(x^*)\}_{k \in \mathbb{N}}$ is nonnegative for all  $k \in \mathbb{N}.$  Since it is also nonincreasing, it follows that it converges. Let $N \in \mathbb{N}_0$. Summing both sides of \eqref{17} over $k = 1,2,\cdots,N$ yields
\begin{equation*}
\begin{aligned}
\sum_{k=1}^N(1- L_{k-1}-L_k-\gamma_k\mu-\gamma_{k+1}\mu-\frac{\gamma_k\beta}{2})\|x_{k+1}-x_k\|_S^2  &\leq \Psi_{1}(x^*)- \Psi_{N+1}(x^*)\\
& \leq \Psi_{1}(x^*) < + \infty.
\end{aligned}  \label{29}
\end{equation*}
Then it follows directly from \eqref{27}  that $\lim\limits_{k\rightarrow \infty}\|x_{k+1}-x_k\|_S=0$. Therefore, we have
\begin{equation*}
	\lim_{k\rightarrow \infty}\Psi_k(x^*)=\lim_{k\rightarrow \infty}\|x_k-x^*\|_S^2. \label{14}
\end{equation*}
This implies that the sequence $\{x_k\}_{k \in \mathbb{N}}$ is bounded. The sequence $\{x_k\}_{k \in \mathbb{N}}$ admits at least one weakly convergent subsequence $\{x_{k_n}\}_{n \in \mathbb{N}}$. Say, without loss of generality, that $x_{k_n} \rightharpoonup \bar{x} \in \mathcal{H}$ as $n\rightarrow \infty$.
Define the sequence $\{\Delta_k\}_{k\in \mathbb{N}}$ by
\begin{equation*}
	\label{delta}
	\Delta_k := M_kx_k-M_kx_{k+1}+(Bx_{k+1}-Bx_k)-(Bx_k-Bx_{k-1})+(Cx_{k+1}-Cx_k)+\gamma_k^{-1}u_k.
\end{equation*}
Then, by \eqref{18}, we have $(x_{k+1}, \Delta_k) \in \gra(A+B+C)$  for all $k \in \mathbb{N}.$ Since $M_k$ is $2\gamma^{-1}$-Lipschitz continuous w.r.t. $S$, $B$ is $\mu$-Lipschitz continuous w.r.t. $S$, and $C$ is $\beta^{-1}$-cocoercive w.r.t. $S$, it follows that
\begin{equation*}
	\aligned \| \Delta_k \|_{S^{-1}}
	 \leq& \|M_kx_k-M_kx_{k+1}\|_{S^{-1}}+\|(B+C)(x_k-x_{k+1})\|_{S^{-1}}+\|Bx_k-Bx_{k-1}\|_{S^{-1}}\\
	&+\gamma_k^{-1}\|u_k\|_{S^{-1}}  \\
	\leq&  (\frac{2}{\gamma}+\mu+\beta)\|x_k-x_{k+1}\|_S+(\mu+\frac{L_{k-1}}{\gamma_k})\|x_{k}-x_{k-1}\|_S.
	\endaligned \label{16}
\end{equation*}
By the arguments presented above, we have shown  that $\Delta_{k_n}\rightarrow 0.$ Furthermore, under the assumption that the operator  $B$ has a full domain, Corollary 25.5 (i) in \cite{BC2011} ensures the maximal monotonicity of  $A+B$. Next, combining Lemma 2.1 in \cite{Showalter} with the assumption that $C$ is cocoercive with respect to $S$, it follows that  $A+B+C$ is maximally monotone. Then, by Lemma \ref{lem3}, it follows that $(\bar{x} , 0) \in \gra(A+B+C)$, which is equivalent to  $\bar{x} \in \zer(A+B+C)$. Finally, by Lemma \ref{lem1},  the sequence $\{x_k\}_{k \in \mathbb{N}}$ converges weakly to a point in $\zer(A+B+C)$.
\end{proof}

\begin{remark}
\label{remark1}
\rm We now consider three special cases of Algorithm \ref{algorithm2}.
\begin{itemize}
\item[{\rm (i)}] When $B=0$, we have $\mu=0$. Then Algorithm \ref{algorithm2} reduces to the nonlinear forward-backward splitting with momentum proposed in \cite{Morin}, and the condition on $\gamma_k$ is recovered as $1-L_{k-1}-L_k-\frac{\gamma_k\beta}{2} \geq \epsilon.$
\item[{\rm (ii)}] When $S=\Id$, $\gamma_k \equiv \gamma$, and $M_k \equiv \frac{\Id}{\gamma}$, $\forall k \in \mathbb{N}$, we have $L_k=0$ and $u_k=0$, $\forall k \in \mathbb{N}$. In this case, Algorithm \ref{algorithm2} simplifies to the SFRBS algorithm \eqref{SFRBS}.  Furthermore, the upper bound on $\gamma$ coincides with that in \cite{Malitsky}, namely, $\gamma < \frac{2}{4\mu+\beta}.$
\item[{\rm (iii)}] Suppose that $A=A_1+A_2,$ where $A_1:\mathcal{H} \rightarrow 2^{\mathcal{H}}$ is maximally monotone, $A_2:\mathcal{H} \rightarrow 2^{\mathcal{H}}$ is $L$-Lipschitz continuous, and $A_1+A_2$ is maximally monotone. Furthermore, let $S=\Id,$ $\gamma_k\equiv \gamma$, and $M_k\equiv M=\frac{\Id}{\gamma}-A_2$, $\forall k \in \mathbb{N}$. In this case, Algorithm \ref{algorithm2} takes the form
\begin{equation}
	\label{factfrb}
		x_{k+1}=J_{\gamma A_1}(x_k-2\gamma A_2x_k-2\gamma Bx_k+\gamma Bx_{k-1}+\gamma A_2x_{k-1}-\gamma Cx_k).
\end{equation}
Since $\gamma M-\Id$ is $\gamma L$-Lipschitz continuous, condition \eqref{27} reduces to
\begin{equation*}
	1-2\gamma L-2\gamma\mu-\frac{\gamma \beta}{2}  \geq \epsilon.
\end{equation*}
In fact, Algorithm \eqref{factfrb} is equivalent to the SFRBS algorithm \eqref{SFRBS} for problem \eqref{ABC}. To see this, define the operator
\begin{equation*}
	D:=A_2+B.
\end{equation*}
Under these assumptions, $D$ is monotone and Lipschitz continuous. Substituting $D$ into \eqref{factfrb} yields  exactly the SFRBS iteration \eqref{SFRBS}, under the same step-size conditions. Moreover, Algorithm \eqref{factfrb} can be viewed as a four-operator splitting method for finding a zero of the sum of one maximally monotone operator, two Lipschitz continuous operators, and one cocoercive operator.
\end{itemize}
\end{remark}

\subsubsection{Linear convergence}
This subsection establishes the $R$-linear convergence of the sequence $\{x_k\}_{k \in \mathbb{N}}$ generated by Algorithm \ref{algorithm2}, under the assumption that the operator $A$ is strongly monotone. Specifically, $A$ is said to be $\rho$-strongly monotone w.r.t. to $S$ ($\rho > 0$) if
\begin{equation*}
	\forall (x,u) \in \gra (A), \forall (y,v) \in \gra (A) \ \Leftrightarrow \ \langle u - v, x - y \rangle \geq \rho \|x - y\|_S^2.
\end{equation*}
Notably, if $B$ is $\rho$-strongly monotone, the operators can be rewritten using the identity operator as $A + B = (A + \rho \Id) + (B - \rho \Id)$. This transformation preserves both monotonicity and Lipschitz continuity. Consequently, the $R$-linear convergence result can be directly extended to the case in which $B$ is strongly monotone, by an analogous proof argument.

\begin{theorem}
	\label{linear1}
	{
		\noindent
		Let Assumption {\rm\ref{assumption1}} hold, and assume that $A$ is $\rho$-strongly monotone. Suppose that condition \eqref{27} holds and there exists $\varepsilon_1 >0$ such that
		\begin{equation*}
t=\min \left \{\frac{2\gamma_k \rho}{1+\frac{L_k}{\varepsilon_1}+\gamma_{k+1}\mu},
\frac{\kappa}{\varepsilon_1 L_k+2\gamma_{k+1}\mu+L_k} \right\},
		\end{equation*}
		where $\kappa=(1-L_{k-1}-L_k-\gamma_k\mu-\gamma_{k+1}\mu-\frac{\gamma_k\beta}{2}).$
		Then the sequence $\{x_k\}_{k \in \mathbb{N}}$ generated by Algorithm \ref{algorithm2} converges $R$-linearly to a point $x^*$ in $\zer(A+B+C)$.
	}
\end{theorem}

\begin{proof}
Using the strong monotonicity of $A$, inequality \eqref{20} can be rewritten as
\begin{equation*}
	\rho\|x_{k+1}-x^*\|_S^2 \leq \langle M_kx_k-M_kx_{k+1}-2Bx_k+Bx_{k-1}-Cx_k+\gamma_k^{-1}u_k+(B+C)x^* , x_{k+1}-x^* \rangle. 		\label{70}		
\end{equation*}
Furthermore, following the same arguments as in Lemma \ref{lemma2}, we obtain the recursive inequality:
\begin{equation}
	\begin{aligned}
		& (1+2\gamma_k\rho) \|x_{k+1}-x^*\|_S^2+2\langle u_{k+1}, x_{k+1}-x^*\rangle+2\gamma_k\langle Bx_{k+1}-Bx_k, x^*-x_{k+1}\rangle\\
		 \leq & \|x_k-x^*\|_S^2 +2\langle u_k,x_k-x^*\rangle+(\gamma_k\mu+L_{k-1})\|x_k-x_{k-1}\|_S^2 \\
		& +2\gamma_k \langle Bx_k-Bx_{k-1},x^*-x_k\rangle -(1-L_{k-1}-\gamma_k\mu-\frac{\gamma_k\beta}{2})\|x_{k+1}-x_k\|_S^2.
	\end{aligned} 					\label{71}
\end{equation}		
Indeed, if $B$ is also $\rho$-strongly monotone, the preceding inequality remains valid. Set
$$a_k(x^*)=2\langle u_k,x_k-x^*\rangle+(\gamma_k\mu+L_{k-1})\|x_{k}-x_{k-1}\|_S^2+2\gamma_k\langle Bx_k-Bx_{k-1},x^*-x_k\rangle, $$
$$b_k=\kappa\|x_{k+1}-x_k\|_S^2.$$
Then inequality \eqref{71} can be equivalently written as
\begin{equation}
	(1+2\gamma_k\rho)\|x_{k+1}-x^*\|_S^2+a_{k+1}(x^*)+b_k \leq \|x_{k}-x^*\|_S^2+a_k(x^*). \label{72}
\end{equation}
Using the Lipschitz continuity of $u_{k+1}$ and Young's inequality, we obtain the following bound for $a_{k+1}(x^*)$:
\begin{equation*}
	\label{73}
	\begin{aligned}
		 a_{k+1}(x^*)
		= & 2\langle u_{k+1},x_{k+1}-x^*\rangle+(\gamma_{k+1}\mu+L_k)\|x_{k+1}-x_{k}\|_S^2\\
		&+2\gamma_{k+1}\langle Bx_{k+1}-Bx_{k},x^*-x_{k+1}\rangle \\
		\leq &  \varepsilon_1 L_k\|x_{k+1}-x_{k}\|_S^2+\frac{L_k}{\varepsilon_1}\|x_{k+1}-x^*\|_S^2+
		(\gamma_{k+1}\mu+L_k)\|x_{k+1}-x_{k}\|_S^2 \\
		&+\gamma_{k+1}\mu\|x_{k+1}-x_{k}\|_S^2+\gamma_{k+1}\mu\|x_{k+1}-x^*\|_S^2\\
		= &  (\varepsilon_1 L_k+2\gamma_{k+1}\mu+L_k)\|x_{k+1}-x_{k}\|_S^2+(\frac{L_k}{\varepsilon_1}+\gamma_{k+1}\mu)\|x_{k+1}-x^*\|^2_S.
	\end{aligned} 					
\end{equation*}
By the assumption of $t$, we obtain that
\begin{equation}
	\label{74}
	ta_{k+1}(x^*) \leq  (2\gamma_k\rho-t)\|x_{k+1}-x^*\|^2_S+b_k.
\end{equation}
By combining inequalities \eqref{72} and \eqref{74}, we deduce that
\begin{equation*}
\label{75}
(1+t)(\|x_{k+1}-x^*\|^2_S+a_{k+1}(x^*) )\leq  \|x_{k}-x^*\|_S^2+a_k(x^*) .
\end{equation*}
On the other hand, from \eqref{28} and condition \eqref{27}, there exists a constant $k_1>0$ such that
\begin{equation*}
	\label{76}
	\|x_{k+1}-x^*\|_S+a_{k+1}(x^*)  \geq (1-L_{k}-\gamma_{k+1}\mu)\|x_{k+1}-x^*\|_S^2\geq k_1\|x_{k+1}-x^*\|_S^2.
\end{equation*}
Therefore,
\begin{equation*}
	\begin{aligned}
		\label{77}
		k_1\|x_{k+1}-x^*\|_S & \leq \|x_{k+1}-x\|_S^2+a_{k+1}(x^*) \\
		&\leq \frac{\|x_{k}-x^*\|_S^2+a_k(x^*) }{1+t}\leq \ldots \leq \frac{\|x_{1}-x^*\|_S^2+a_1(x^*) }{(1+t)^k}.
	\end{aligned}
\end{equation*}
Hence, the sequence $\{x_k\}_{k \in \mathbb{N}}$ generated by Algorithm \ref{algorithm2} converges $R$-linearly to a point $x^*$ in $\zer(A+B+C)$.
\end{proof}

\subsection{Nonlinear semi-reflected-forward-backward splitting algorithm with momentum}
In this subsection, we propose a nonlinear semi-reflected-forward-backward splitting (SRFBS) algorithm with momentum, denoted as Algorithm \ref{algorithm3},  and establish its convergence results.
\begin{algorithm}\label{algorithm3}
\hrule
\noindent\textbf{\footnotesize{Nonlinear SRFBS algorithm with momentum}}
\hrule
\vskip 1mm
\noindent Let $S \in \mathcal{P}(\mathcal{H})$, $M_k: \mathcal{H} \rightarrow \mathcal{H}$ $(\forall k \in\mathbb{N})$,  $\gamma > 0$ and initial point $x_{0}, \ u_{0} \in \mathcal{H}.$ For $k=0,1,\cdots,$ do\\
\begin{equation*}
\label{algorithm_3}
\aligned
&y_k=2x_k-x_{k-1} \\
&x_{k+1}=(M_k+A)^{-1}(M_kx_k-By_k-Cx_k+\gamma^{-1}u_k), \\
&u_{k+1}=(\gamma M_k-S)x_{k+1}-(\gamma M_k-S)x_k.
\endaligned
\end{equation*}
\vskip 1mm
\hrule
\hspace*{\fill}
\end{algorithm}

For the iterates $\{x_k\}_{k \in \mathbb{N}}$, $\{y_k\}_{k \in \mathbb{N}}$, and $\{u_k\}_{k \in \mathbb{N}}$ generated by Algorithm \ref{algorithm3}, and $ \forall x \in \zer(A+B+C),$ we define
\begin{equation}\label{Gammak}
\begin{aligned}
\Gamma_{k}(x):=&\|x_k-x\|_{S}^2+2\langle u_k, x_k-x\rangle +2\gamma \langle By_{k-1}-Bx, x_k-x_{k-1}\rangle \\
&+2\langle u_k-u_{k-1}, x_k-x_{k-1}\rangle,
\end{aligned}
\end{equation}
which will be used to establish the weak convergence of Algorithm \ref{algorithm3}.	

\begin{lemma}
\label{lemma3}
Suppose that Assumption {\rm\ref{assumption1}} holds with $\gamma_k \equiv \gamma$ for all $k \in \mathbb{N}$, and let  $\varepsilon_2>0$. Then, for the sequences  $\{x_k\}_{k \in \mathbb{N}},$ $\{y_k\}_{k \in \mathbb{N}},$ and $\{u_k\}_{k \in \mathbb{N}}$ generated by Algorithm \ref{algorithm3}, and $\forall x^* \in \zer(A+B+C),$ the following inequality holds:
\begin{equation}
\begin{aligned}
	& \Gamma_{k+1}(x^*)+(2-\frac{(1+2\varepsilon_2)\gamma\beta}{2}-L_{k-1})\|x_{k+1}-x_k\|_S^2 +\|x_{k+1}-y_k\|_S^2 \\
	\leq &  \Gamma_{k}(x^*)+(1+2L_{k-1}+\frac{\gamma\beta}{\varepsilon_2}+\gamma\mu(\sqrt{2}+1))\|x_k-x_{k-1}\|_S^2 +L_{k-2}\|x_{k-1}-x_{k-2}\|_S^2\\
	&+\gamma\mu\|x_k-y_{k-1}\|_S^2+(\sqrt{2}\gamma\mu+L_{k-1}+L_{k-2})\|x_{k+1}-y_{k}\|_S^2
 \label{30}
\end{aligned}
\end{equation}
for all $n \in \mathbb{N}_0.$
\end{lemma}
			
\begin{proof}
From Algorithm \ref{algorithm3}, we have
$$x_{k+1}=(M_k+A)^{-1}(M_kx_k-By_k-Cx_k+\gamma^{-1}u_k),$$
which is equivalent to the inclusion
\begin{equation}
M_kx_k-M_kx_{k+1}-By_k-Cx_k+\gamma^{-1}u_k \in Ax_{k+1}.     \label{31}
\end{equation}
Since $x^* \in \zer(A+B+C),$ we have
\begin{equation}
-(B+C)x^* \in Ax^*.                                 \label{32}
\end{equation}
Combining \eqref{31}, \eqref{32}, and using the monotonicity of $A$, we obtain
\begin{equation*}
0 \leq \langle M_kx_k-M_kx_{k+1}-By_k-Cx_k+\gamma^{-1}u_k+(B+C)x^* , x_{k+1}-x^* \rangle. \label{33}
\end{equation*}
Multiplying both sides by $2\gamma$ and applying the definition of $u_{k+1}$ yield
\begin{equation}
\aligned
0& \leq 2\langle Sx_k-Sx_{k+1}+u_k-u_{k+1}-\gamma By_k-\gamma Cx_k+\gamma (B+C)x^* , x_{k+1}-x^* \rangle\\
& = 2\langle Sx_k-Sx_{k+1}, x_{k+1}-x^* \rangle +2\langle u_k-u_{k+1},x_{k+1}-x^* \rangle-2\gamma \langle By_k-Bx^*,x_{k+1}-x^*\rangle \\
& \quad -2\gamma\langle Cx_k-Cx^*,x_{k+1}-x^*\rangle. 	
\endaligned
\label{34}
\end{equation}
In the following, we analyze each term on the right-hand side of the inequality \eqref{34} separately.  Analogous to \eqref{22} and \eqref{23}, we obtain
				\begin{equation}
					\begin{aligned}
					   & 2\langle Sx_k -Sx_{k+1}, x_{k+1}-x^*\rangle \\
                       =&\|x_k-x^*\|_S^2-\|x_{k+1}-x^*\|_S^2-\|x_{k+1}-x_k\|_S^2,
					\end{aligned}  \label{35}
				\end{equation}
				and
				\begin{equation}
					\begin{aligned}
						& 2\langle u_k-u_{k+1}, x_{k+1}-x^*\rangle \\
                        =&2\langle u_k,x_k-x^*\rangle+2\langle u_k,x_{k+1}-x_k\rangle-2\langle u_{k+1},x_{k+1}-x^*\rangle  \\
                        \leq & 2\langle u_k,x_k-x^*\rangle-2\langle u_{k+1},x_{k+1}-x^*\rangle+L_{k-1}\|x_k-x_{k-1}\|_S^2 \\
                        &+L_{k-1}\|x_{k+1}-x_k\|_S^2.
					\end{aligned}  \label{36}
				\end{equation}
The monotonicity of  $B:\mathcal{H} \rightarrow \mathcal{H}$ implies
\begin{equation}
\begin{aligned}
&-2\gamma \langle By_k-Bx^*, x_{k+1}-x^*\rangle\\
=&-2\gamma \langle By_k-Bx^*, x_{k+1}-y_k\rangle-2\gamma \langle By_k-Bx^*, y_k-x^*\rangle \\
\leq & -2\gamma \langle By_k-Bx^*, x_{k+1}-y_k\rangle \\
=& -2\gamma \langle By_k-Bx^*, x_{k+1}-x_k\rangle + 2\gamma \langle By_k-Bx^*, x_k-x_{k-1}\rangle \\
=& -2\gamma \langle By_k-Bx^*, x_{k+1}-x_k\rangle+2\gamma \langle By_{k-1}-Bx^*, x_k-x_{k-1}\rangle \\
&+2\gamma \langle By_k-By_{k-1}, x_k-x_{k-1}\rangle \\
=& -2\gamma \langle By_k-Bx^*, x_{k+1}-x_k\rangle+2\gamma \langle By_{k-1}-Bx^*, x_k-x_{k-1}\rangle \\
&+2\gamma \langle By_k-By_{k-1},x_{k+1}-x_k\rangle+2\gamma \langle By_k-By_{k-1},y_k-x_{k+1}\rangle.
\end{aligned}  \label{37}
\end{equation}
Moreover, by Lemma \ref{lem0}, the cocoercivity of $C$ yields
\begin{equation}
-2\gamma\langle Cx_k-Cx^* , x_{k+1}-x^* \rangle \leq \frac{\gamma\beta}{2}\|x_{k+1}-x_k\|_S^2. \label{38}
\end{equation}
Combining \eqref{34}--\eqref{38}, we obtain
\begin{equation}
\begin{aligned}
& \|x_{k+1}-x^*\|_S^2+2\langle u_{k+1}, x_{k+1}-x^*\rangle+2\gamma\langle By_k-Bx^*, x_{k+1}-x_k\rangle\\
& +(1-\frac{\gamma\beta}{2}-L_{k-1})\|x_{k+1}-x_k\|_S^2 \\
\leq & \|x_k-x^*\|_S^2 +2\langle u_k,x_k-x^*\rangle +L_{k-1}\|x_k-x_{k-1}\|_S^2\\
& +2\gamma\langle By_{k-1}-Bx^*, x_k-x_{k-1}\rangle +2\gamma \langle By_k-By_{k-1},x_{k+1}-x_k\rangle \\
&+2\gamma \langle By_k-By_{k-1},y_k-x_{k+1}\rangle.
\end{aligned} 					\label{39}
\end{equation}	
Furthermore, from Algorithm \ref{algorithm3} we have
\begin{equation*}
\begin{aligned}
& M_kx_k-M_kx_{k+1}-By_k-Cx_k+\gamma^{-1}u_k \in Ax_{k+1}, \\
& M_{k-1}x_{k-1}-M_{k-1}x_k-By_{k-1}-Cx_{k-1}+\gamma^{-1}u_{k-1} \in Ax_k.
\end{aligned} 					\label{40}
\end{equation*}	
By the monotonicity of $A$, it follows that
\begin{equation*}
\begin{aligned}
&\langle M_kx_k-M_kx_{k+1}-By_k-Cx_k+\gamma^{-1}u_k,x_{k+1}-x_k \rangle \\
&- \langle M_{k-1}x_{k-1}-M_{k-1}x_k-By_{k-1}
-Cx_{k-1}+\gamma^{-1}u_{k-1},x_{k+1}-x_k \rangle \geq  0. \label{41}
\end{aligned}
\end{equation*}	
Multiplying both sides by $2\gamma$ and applying the Young's inequality yield
\begin{equation}
\begin{aligned}
&2\gamma \langle By_k-By_{k-1}, x_{k+1}-x_k\rangle \\
\leq & 2\langle Sx_k-Sx_{k+1}-u_{k+1}+u_k-Sx_{k-1}+Sx_k+u_k-u_{k-1}-\gamma Cx_k+\gamma Cx_{k-1}, x_{k+1}-x_k\rangle \\
= & -2\langle x_{k+1}-x_k, x_{k+1}-x_k\rangle_S+2\langle x_k-x_{k-1},x_{k+1}-x_k\rangle_S \\
&+2\langle -u_{k+1}+2u_k-u_{k-1}, x_{k+1}-x_k\rangle -2\gamma\langle Cx_k-Cx_{k-1},x_{k+1}-x_k \rangle \\
= &-2 \|x_{k+1}-x_k\|_S^2+2\langle y_k-x_k, x_{k+1}-x_k\rangle_S+2\langle -u_{k+1}+2u_k-u_{k-1}, x_{k+1}-x_k\rangle \\
& -2\gamma\langle Cx_k-Cx_{k-1},x_{k+1}-x_k \rangle \\
\leq &- \|x_{k+1}-x_k\|_S^2- \|x_{k+1}-y_k\|_S^2+ \|x_k-x_{k-1}\|_S^2-2\langle u_{k+1}-u_k, x_{k+1}-x_k\rangle \\
& +2\langle u_k-u_{k-1}, x_{k+1}-x_k\rangle +\frac{\gamma\beta}{\varepsilon_2} \|x_k-x_{k-1}\|_S^2+\gamma\beta\varepsilon_2 \|x_{k+1}-x_k\|_S^2 \\
=&- \|x_{k+1}-x_k\|_S^2- \|x_{k+1}-y_k\|_S^2+ \|x_k-x_{k-1}\|_S^2-2\langle u_{k+1}-u_k, x_{k+1}-x_k\rangle \\
& +2\langle u_k-u_{k-1}, x_k-x_{k-1}\rangle + +2\langle u_k-u_{k-1}, x_{k+1}-y_k\rangle +\frac{\gamma\beta}{\varepsilon_2} \|x_k-x_{k-1}\|_S^2 \\
&+\gamma\beta\varepsilon_2 \|x_{k+1}-x_k\|_S^2.
\label{42}
\end{aligned}
\end{equation}
Using the  Lipschitz property of $B$, we have
\begin{equation}
\begin{aligned}
& 2\gamma \langle By_k-By_{k-1},y_k-x_{k+1}\rangle \\
\leq &2\gamma\mu \|y_k-y_{k-1}\|_S\|x_{k+1}-y_k\|_S \\
\leq &2\gamma\mu (\|x_k-y_k\|_S+\|x_k-y_{k-1}\|_S)\|x_{k+1}-y_k\|_S \\
\leq &\gamma\mu ((\sqrt{2}+1)\|x_k-x_{k-1}\|_S^2+(\sqrt{2}-1)\|x_{k+1}-y_k\|_S^2)\\
& +\gamma\mu(\|x_k-y_{k-1}\|_S^2+\|x_{k+1}-y_k\|_S^2) \\
= &\gamma\mu ((\sqrt{2}+1)\|x_k-x_{k-1}\|_S^2+\sqrt{2}\|x_{k+1}-y_k\|_S^2)+\gamma\mu\|x_k-y_{k-1}\|_S^2.   \label{43}
\end{aligned}
\end{equation}
Combining \eqref{39}, \eqref{42} and \eqref{43}, we obtain
\begin{equation*}
\begin{aligned}
&\|x_{k+1}-x^*\|_S^2+2\langle u_{k+1}, x_{k+1}-x^*\rangle+2\gamma\langle By_k-Bx^*, x_{k+1}-x_k\rangle\\
& +2\langle u_{k+1}-u_k, x_{k+1}-x_k\rangle +(2-\frac{(1+2\varepsilon_2)\gamma\beta}{2}-L_{k-1})\|x_{k+1}-x_k\|_S^2 +\|x_{k+1}-y_k\|_S^2 \\
\leq & \|x_k-x^*\|_S^2 +2\langle u_k,x_k-x^*\rangle+2\gamma\langle By_{k-1}-Bx^*, x_k-x_{k-1}\rangle\\
& +2\langle u_k-u_{k-1}, x_k-x_{k-1}\rangle +(1+L_{k-1}+\frac{\gamma\beta}{\varepsilon_2}+\gamma\mu(\sqrt{2}+1) )\|x_k-x_{k-1}\|_S^2 \\
& +2\langle u_k-u_{k-1}, x_{k+1}-y_k\rangle+\sqrt{2}\gamma\mu\|x_{k+1}-y_k\|_S^2+\gamma\mu\|x_k-y_{k-1}\|_S^2. \\
\end{aligned} 					\label{44}
\end{equation*}	
Using the  Lipschitz property of  $\gamma M_k - S$ for all $k \in \mathbb{N}$, we obtain
\begin{equation*}
\begin{aligned}
& 2\langle u_k-u_{k-1}, x_{k+1}-y_k\rangle\\
\leq & L_{k-1}\|x_k-x_{k-1}\|_S^2+L_{k-2}\|x_{k-1}-x_{k-2}\|_S^2+(L_{k-1}+L_{k-2})\|x_{k+1}-y_{k}\|_S^2. \\
\end{aligned} 					\label{45}
\end{equation*}	
Therefore,
\begin{equation*}
\begin{aligned}
&\|x_{k+1}-x^*\|_S^2+2\langle u_{k+1}, x_{k+1}-x^*\rangle+2\gamma\langle By_k-Bx^*, x_{k+1}-x_k\rangle\\
& +2\langle u_{k+1}-u_k, x_{k+1}-x_k\rangle +(2-\frac{(1+2\varepsilon_2)\gamma\beta}{2}-L_{k-1})\|x_{k+1}-x_k\|_S^2 +\|x_{k+1}-y_k\|_S^2 \\
\leq & \|x_k-x^*\|_S^2 +2\langle u_k,x_k-x^*\rangle+2\gamma\langle By_{k-1}-Bx^*, x_k-x_{k-1}\rangle\\
& +2\langle u_k-u_{k-1}, x_k-x_{k-1}\rangle+(1+2L_{k-1}+\frac{\gamma\beta}{\varepsilon_2}+\gamma\mu(\sqrt{2}+1))\|x_k-x_{k-1}\|_S^2 \\
&+L_{k-2}\|x_{k-1}-x_{k-2}\|_S^2+\gamma\mu\|x_k-y_{k-1}\|_S^2+(\sqrt{2}\gamma\mu+L_{k-1}+L_{k-2})\|x_{k+1}-y_{k}\|_S^2. \\
\end{aligned} 					\label{46}
\end{equation*}	
Then, \eqref{30} is deduced by the definition of $\Gamma_{k}(x)$ in \eqref{Gammak}.
\end{proof}

\begin{theorem}
\label{theorem3}
{
\noindent
Suppose that Assumption {\rm\ref{assumption1}} holds with $\gamma_k \equiv \gamma$ for all $k \in \mathbb{N}$, and let $\varepsilon_2>0$. Assume further that there exists a constant $\epsilon >0$ such that
\begin{equation}
\begin{split}
&1-3L_k-\frac{(2+\varepsilon_2+2\varepsilon_2^2)\gamma\beta}{2\varepsilon_2}-L_{k-1}-\gamma\mu(\sqrt{2}+1)\geq \epsilon, \\
&1-L_{k-1}-L_{k-2}-\gamma\mu(\sqrt{2}+1)\geq \epsilon\label{47}
\end{split}
\end{equation}
for all $n \in \mathbb{N}_0$. Then, the sequence $\{x_k\}_{k \in \mathbb{N}}$ generated by Algorithm \ref{algorithm3} converges weakly to a point $x^*$ in $\zer(A+B+C)$.
}
\end{theorem}

\begin{proof}
For any $x^*$ in $\zer(A+B+C)$, define
\begin{equation}\label{Xik}
\begin{aligned}
\Xi_{k}(x^*):=&\Gamma_k(x^*)+(1+3L_{k-1}+\frac{\gamma\beta}{\varepsilon_2}+\gamma\mu(\sqrt{2}+1))\|x_k-x_{k-1}\|_S^2 \\
&+L_{k-2}\|x_{k-1}-x_{k-2}\|_S^2+\gamma\mu\|x_k-y_{k-1}\|_S^2.
\end{aligned}
\end{equation}
Then,  using inequality \eqref{30}, the following recursive relation holds:
\begin{equation}
\begin{aligned}
\Xi_{k+1}(x^*)\leq &\Xi_{k}(x^*)-(1-3L_k-\frac{(2+\varepsilon_2+2\varepsilon_2^2)\gamma\beta}{2\varepsilon_2}-L_{k-1}-\gamma\mu(\sqrt{2}+1))\|x_{k+1}-x_k\|_S^2 \\
&-(1-L_{k-1}-L_{k-2}-\gamma\mu(\sqrt{2}+1))\|x_{k+1}-y_k\|_S^2.
\end{aligned} 					\label{48}
\end{equation}
Next, we show that $\Xi_{k}(x^*)$ is bounded from below, $\forall k \in \mathbb{N}.$
We first analyze $\Gamma_k(x^*)$ using Assumption {\rm\ref{assumption1}} and the Cauchy-Schwarz inequality:
\begin{equation}
\begin{aligned}
\Gamma_{k}(x^*):=&\|x_k-x^*\|_{S}^2+2\langle u_k, x_k-x^*\rangle +2\gamma \langle By_{k-1}-Bx^*, x_k-x_{k-1}\rangle \\
&+2\langle u_k-u_{k-1}, x_k-x_{k-1}\rangle \\
\geq & \|x_k-x^*\|_{S}^2-2\|u_k\|_{S^{-1}}\|x_k-x^*\|_S+2\gamma \langle By_{k-1}-Bx_k, x_k-x_{k-1}\rangle \\
&+2\gamma \langle Bx_k-Bx^*, x_k-x_{k-1}\rangle -2\|u_k\|_{S^{-1}}\|x_k-x_{k-1}\|_S \\
& -2\|u_{k-1}\|_{S^{-1}}\|x_k-x_{k-1}\|_S \\
\geq & \|x_k-x^*\|_{S}^2-L_{k-1}\|x_k-x_{k-1}\|_S^2-L_{k-1}\|x_k-x^*\|_S^2\\
& -\gamma\mu\|x_k-y_{k-1}\|_S^2-2\gamma\mu\|x_k-x_{k-1}\|_S^2-\gamma\mu\|x_k-x^*\|_S^2\\
& -(2L_{k-1}+L_{k-2})\|x_k-x_{k-1}\|_S^2-L_{k-2}\|x_{k-1}-x_{k-2}\|_S^2 \\
=& (1-L_{k-1}-\gamma\mu)\|x_k-x^*\|_{S}^2-(3L_{k-1}+2\gamma\mu+L_{k-2})\|x_k-x_{k-1}\|_S^2 \\
& -\gamma\mu\|x_k-y_{k-1}\|_S^2-L_{k-2}\|x_{k-1}-x_{k-2}\|_S^2. \end{aligned} 					\label{49}
\end{equation}
Therefore, by combining the lower bound of $\Gamma_k(x^*)$ in \eqref{49} with the definition of $\Xi_k(x^*)$ in \eqref{Xik}, we obtain the following estimate:
\begin{equation}
\begin{aligned}
\Xi_k(x^*)\geq& (1-L_{k-1}-\gamma\mu)\|x_k-x^*\|_{S}^2-(3L_{k-1}+2\gamma\mu+L_{k-2})\|x_k-x_{k-1}\|_S^2 \\
&-\gamma\mu\|x_k-y_{k-1}\|_S^2-L_{k-2}\|x_{k-1}-x_{k-2}\|_S^2+L_{k-2}\|x_{k-1}-x_{k-2}\|_S^2\\
&+\gamma\mu\|x_k-y_{k-1}\|_S^2 +(1+3L_{k-1}+\frac{\gamma\beta}{\varepsilon_2}+\gamma\mu(\sqrt{2}+1))\|x_k-x_{k-1}\|_S^2  \\
=&(1-L_{k-1}-\gamma\mu)\|x_k-x^*\|_{S}^2+(1+\frac{\gamma\beta}{\varepsilon_2}-L_{k-2}+(\sqrt{2}-1)\gamma\mu)\|x_k-x_{k-1}\|_S^2 \\
\geq &0.    \label{50}
\end{aligned}
\end{equation}
Hence, the sequence $\{\Xi_{k}(x^*)\}_{k \in \mathbb{N}}$ is nonnegative, $\forall k \in \mathbb{N}.$  Since it is also nonincreasing by \eqref{48}, it follows that $\{\Xi_{k}(x^*)\}$ is convergent. Moreover, from inequality \eqref{48}, it follows that
$$\lim\limits_{k\rightarrow \infty}\|x_{k+1}-x_k\|_S=0,
$$
and
$$\lim\limits_{k\rightarrow \infty}\|x_{k+1}-y_k\|_S=0.$$ Consequently, by the definition of $\Xi_k(x^*)$, we have
\begin{equation*}
	\lim_{k\rightarrow \infty}\Xi_k(x^*)=\lim_{k\rightarrow \infty}\|x_k-x^*\|_S^2,  \label{14}
\end{equation*}
which implies that the sequence $\{x_k\}_{k \in \mathbb{N}}$ is bounded. Hence, there exists at least one weakly convergent subsequence $\{x_{k_n}\}_{n \in \mathbb{N}}$ such that  $x_{k_n} \rightharpoonup \bar{x} \in \mathcal{H}$ as $n\rightarrow \infty$.
Define
\begin{equation*}
	\label{delta}
	\Delta_k := M_kx_k-M_kx_{k+1}+(Bx_{k+1}-By_k)+(Cx_{k+1}-Cx_k)+\gamma^{-1}u_k.
\end{equation*}
Then, from \eqref{31}, we have $(x_{k+1}, \Delta_k) \in \gra(A+B+C)$ for all $k \in \mathbb{N}.$ Using the facts that $M_k$ is $2\gamma^{-1}$-Lipschitz continuous w.r.t. $S$, $B$ is $\mu$-Lipschitz continuous w.r.t. $S$, and $C$ is $\beta^{-1}$-cocoercive w.r.t. $S$, we obtain
\begin{equation*}
	\aligned \| \Delta_k \|_{S^{-1}}
	\leq& \|M_kx_k-M_kx_{k+1}\|_{S^{-1}}+\|Bx_{k+1}-By_k\|_{S^{-1}}+\|Cx_{k+1}-Cx_k\|_{S^{-1}}\\
	&+\gamma^{-1}\|u_k\|_{S^{-1}}  \\
	\leq&  (\frac{2}{\gamma}+\beta)\|x_k-x_{k+1}\|_S+\mu\|x_{k+1}-y_k\|_S+\frac{L_{k-1}}{\gamma}\|x_{k}-x_{k-1}\|_S.
	\endaligned \label{16}
\end{equation*}
By the arguments above, we have $\Delta_{k_n}\rightarrow 0.$
The remaining part of the proof follows the same reasoning as in Theorem \ref{theorem2} and is therefore omitted.
\end{proof}

\begin{remark}
\label{remark2}
	\rm We present two special cases of Algorithm \ref{algorithm3}.
	\begin{itemize}
		\item[{\rm (i)}] When $S=\Id$, $\gamma_k \equiv \gamma$, and $M_k\equiv\frac{\Id}{\gamma}$, we have $L_k=0$ and $u_k=0,$ $\forall k \in \mathbb{N}$. In this case, Algorithm \ref{algorithm3} reduces to the SRFBS algorithm \eqref{SRFBS}. Moreover, the admissible step size of $\gamma$ satisfies $\frac{1}{\frac{(2+\varepsilon_2+2\varepsilon_2^2)\beta}{2\varepsilon_2}+\mu(\sqrt{2}+1)}.$ Notably, this upper bound on the step size is less restrictive (i.e., allows a larger $\gamma$) than the corrsponding condition in \cite{Cevher2019}.

		\item[{\rm (ii)}] Suppose that $A=A_1+A_2,$ where $A_1:\mathcal{H} \rightarrow 2^{\mathcal{H}}$ is maximally monotone, $A_2:\mathcal{H} \rightarrow 2^{\mathcal{H}}$ is $L$-Lipschitz continuous, and $A_1+A_2$ is maximally monotone. Let $S=\Id,$ $\gamma_k\equiv \gamma$, and $M_k\equiv M=\frac{\Id}{\gamma}-A_2$, $\forall k \in \mathbb{N}$. Then Algorithm \ref{algorithm3} takes the form
			\begin{equation}
            \label{remark2+}
\left\{
\begin{array}{lr}
	y_k=2x_k-x_{k-1}, & \\
	x_{k+1}=J_{\gamma A_1}(x_k-2\gamma A_2x_k-\gamma By_k-\gamma Cy_k+A_2x_{k-1}). & \\
\end{array}
\right.		
			\end{equation}
            Since $\gamma M-\Id$ is $\gamma L$-Lipschitz continuous, condition \eqref{47} becomes
		\begin{equation*}
			1-4\gamma L-\frac{(2+\varepsilon_2+2\varepsilon_2^2)\gamma\beta}{2\varepsilon_2}-\gamma\mu(\sqrt{2}+1)\geq \epsilon.
		\end{equation*}
	\end{itemize}
\end{remark}

\subsubsection{Linear convergence}
This subsection establishes the $R$-linear convergence of the sequence $\{x_k\}_{k \in \mathbb{N}}$ generated by Algorithm \ref{algorithm3} under the assumption that $A$ is strongly monotone.

\begin{theorem}
\label{linear2}
{
\noindent
Let Assumption {\rm\ref{assumption1}} hold, and suppose that $\gamma_k \equiv \gamma$ for all $k \in \mathbb{N}$. Assume further that $A$ is $\rho$-strongly monotone, and that the conditions in \eqref{47} are satisfied. If there exist constants $\varepsilon_3, \varepsilon_4 >0$ such that
\begin{equation}
\label{111}
\scriptsize
\begin{aligned}
& t= \\
& \min\left\{\frac{2\gamma \rho}{1+\frac{L_k}{ \varepsilon_3}+2\varepsilon_4\gamma\mu},
 \frac{p}{1+5L_{k-1}+(\varepsilon_3+5)L_k+\gamma\beta+\gamma\mu(\frac{1}{\varepsilon_4}+\sqrt{2}+1)}, \frac{q}{\gamma\mu(2\varepsilon_4+1)+4L_{k-1}}\right\},
 \end{aligned}
\end{equation}
where $p=1-3L_k-\frac{(2+\varepsilon_2+2\varepsilon_2^2)\gamma\beta}{2\varepsilon_2}-L_{k-1}-\gamma\mu(\sqrt{2}+1),$ $q=1-L_{k-1}-L_{k-2}-\gamma\mu(\sqrt{2}+1),$
then the sequence $\{x_k\}_{k \in \mathbb{N}}$ generated by Algorithm \ref{algorithm3} converges $R$-linearly to a point $x^*$ in $\zer(A+B+C)$.
}
\end{theorem}
\begin{proof}
For any $x^*$ in $\zer(A+B+C)$, define
	\begin{equation*}
		\label{108}
		\begin{aligned}
			c_{k}(x^*)=&2\langle u_k, x_k-x^*\rangle +2\gamma \langle By_{k-1}-Bx^*, x_k-x_{k-1}\rangle+2\langle u_k-u_{k-1}, x_k-x_{k-1}\rangle \\
			&+(1+3L_{k-1}+\frac{\gamma\beta}{\varepsilon_2}+\gamma\mu(\sqrt{2}+1))\|x_k-x_{k-1}\|_S^2+L_{k-2}\|x_{k-1}-x_{k-2}\|_S^2\\
			&+\gamma\mu\|x_k-y_{k-1}\|_S^2.
		\end{aligned}
	\end{equation*}
	By applying the strong monotonicity of $A$ in Lemma \ref{lemma3} and following the reason in Theorem \ref{theorem3}, we obtain the following inequality:
	\begin{equation}
		\label{109}
		(1+2\gamma\rho)\|x_{k+1}-x^*\|_S^2+c_{k+1}(x^*)+p\|x_{k+1}-x_k\|_S^2+q\|x_{k+1}-y_k\|_S^2\leq  \|x_k-x^*\|_S^2 +c_k(x^*).				
	\end{equation}
By applying the Lipschitz continuity of $B$ and $u_{k+1}$ together with Young's inequality, we obtain
	\begin{equation*}
		\label{110}
		\begin{aligned}
			c_{k+1}(x^*)=&2\langle u_{k+1}, x_{k+1}-x^*\rangle +2\gamma \langle By_k-Bx^*, x_{k+1}-x_k\rangle+2\langle u_{k+1}-u_{k}, x_{k+1}-x_{k}\rangle \\
			&+(1+3L_{k}+\frac{\gamma\beta}{\varepsilon_2}+\gamma\mu(\sqrt{2}+1))\|x_{k+1}-x_k\|_S^2+L_{k-1}\|x_{k}-x_{k-1}\|_S^2\\
			&+\gamma\mu\|x_{k+1}-y_{k}\|_S^2\\
			\leq &\varepsilon_3L_k\| x_{k+1}-x_k\|_S^2+\frac{L_k}{\varepsilon_3}\| x_{k+1}-x^*\|_S^2+\varepsilon_4\gamma\mu\| y_{k}-x^*\|_S^2+\frac{\gamma\mu}{\varepsilon_4}\|x_{k+1}-x_{k}\|_S^2 \\
			&+2L_k\|x_{k+1}-x_{k}\|_S^2+2L_{k-1}\|x_{k}-x_{k-1}\|_S^2+L_{k-1}\|x_{k+1}-x_{k}\|_S^2 \\
			&+(1+3L_{k}+\frac{\gamma\beta}{\varepsilon_2}+\gamma\mu(\sqrt{2}+1))\|x_{k+1}-x_k\|_S^2+\gamma\mu\|x_{k+1}-y_{k}\|_S^2\\
			\leq &\varepsilon_3L_k\| x_{k+1}-x_k\|_S^2+\frac{L_k}{\varepsilon_3}\| x_{k+1}-x^*\|_S^2+2\varepsilon_4\gamma\mu\| x_{k+1}-y_{k}\|_S^2 \\
			&+2\varepsilon_4\gamma\mu\| x_{k+1}-x^*\|_S^2+\frac{\gamma\mu}{\varepsilon_4}\|x_{k+1}-x_{k}\|_S^2+2L_k\|x_{k+1}-x_{k}\|_S^2 \\ &+2L_{k-1}\|(x_{k+1}-x_{k})-(x_{k+1}-y_k)\|_S^2 +L_{k-1}\|x_{k+1}-x_{k}\|_S^2 \\
			&+(1+3L_{k}+\frac{\gamma\beta}{\varepsilon_2}+\gamma\mu(\sqrt{2}+1))\|x_{k+1}-x_k\|_S^2+\gamma\mu\|x_{k+1}-y_{k}\|_S^2\\
			\leq &(\frac{L_k}{\varepsilon_3}+2\varepsilon_4\gamma\mu)\| x_{k+1}-x^*\|_S^2 +(\gamma\mu(2\varepsilon_4+1)+4L_{k-1})\| x_{k+1}-y_{k}\|_S^2\\
			&+(1+5L_{k-1}+(\varepsilon_3+5)L_k+\frac{\gamma\beta}{\varepsilon_2}+\gamma\mu(\frac{1}{\varepsilon_4}+\sqrt{2}+1))\|x_{k+1}-x_k\|_S^2. \\
		\end{aligned}
	\end{equation*}
Observe from the definition of $t$ in \eqref{111} that
	\begin{equation}
		\label{112}
		tc_{k+1}(x^*) \leq  (2\gamma\rho-t)\|x_{k+1}-x^*\|_S+p\|x_{k+1}-x_k\|_S^2+q\| x_{k+1}-y_{k}\|_S^2.
	\end{equation}	
Combining \eqref{109} and \eqref{112} gives
	\begin{equation*}
			\label{105}
(1+t)(\|x_{k+1}-x^*\|_S+c_{k+1}(x^*) ) \leq \|x_k-x^*\|_S^2 +c_k(x^*) 	
	\end{equation*}
	Furthermore, by \eqref{50} and under conditions \eqref{47}, there exists a constant $k_2>0$ such that
	\begin{equation*}
		\begin{aligned}
			\label{106}
			&\|x_{k+1}-x^*\|_S+c_{k+1}(x^*)  \\
			\geq &(1-L_{k}-\gamma\mu)\|x_{k+1}-x^*\|_{S}^2+(1-L_{k-1}+\frac{\gamma\beta}{\varepsilon_2}+(\sqrt{2}-1)\gamma\mu)\|x_{k+1}-x_k\|_S^2 \\
			\geq & k_2\|x_{k+1}-x^*\|_S^2.
		\end{aligned} 	
	\end{equation*}
	Therefore,
	\begin{equation*}
		\begin{aligned}
			\label{107}
			k_2\|x_{k+1}-x^*\|_S & \leq \|x_{k+1}-x\|_S^2+c_{k+1}(x^*) \\
			&\leq \frac{\|x_{k}-x^*\|_S^2+c_k(x^*)}{1+t}\leq \ldots \leq \frac{\|x_{1}-x^*\|_S^2+c_1(x^*)}{(1+t)^k},
		\end{aligned}
	\end{equation*}
	which proves that the sequence $\{x_k\}_{k \in \mathbb{N}}$ generated by Algorithm \ref{algorithm3} converges $R$-linearly to a point $x^*$ in $\zer(A+B+C)$.
\end{proof}

\subsection{Nonlinear outer reflected forward-backward splitting algorithm with momentum}
In this subsection, we introduce a nonlinear outer reflected forward-backward splitting (ORFBS) algorithm with momentum, denoted by Algorithm \ref{algorithm4},  and establish its convergence analysis.
\begin{algorithm}\label{algorithm4}
				\hrule
				\noindent\textbf{\footnotesize{Nonlinear ORFBS algorithm with momentum}}
				\hrule
				
				\vskip 1mm
\noindent Let $S \in \mathcal{P}(\mathcal{H})$, $M_k: \mathcal{H} \rightarrow \mathcal{H}$ $(\forall k \in\mathbb{N})$,  $\gamma > 0$ and initial point $x_{0}, \ u_{0} \in \mathcal{H}.$ For $k=0,1,\cdots,$ do \\
\begin{equation*}
\label{algorithm_4}
\aligned
&y_k=(M_k+A)^{-1}(M_kx_k-(B+C)x_k+\gamma^{-1}u_k), \\
&x_{k+1}=y_k-\gamma S^{-1}Bx_k+\gamma S^{-1}Bx_{k-1}, \\
&u_{k+1}=(\gamma M_k-S)y_k-(\gamma M_k-S)x_k.
\endaligned
\end{equation*}
				\vskip 1mm
				
				\hrule
				
				\hspace*{\fill}
			\end{algorithm}
For the iterates $\{x_k\}_{k \in \mathbb{N}}$, $\{y_k\}_{k \in \mathbb{N}}$ and $\{u_k\}_{k \in \mathbb{N}}$ generated by Algorithm \ref{algorithm4},  $ \forall x \in \zer(A+B+C),$ $\varepsilon_5, \ \varepsilon_6>0 $  and $\alpha>0,$ we define the following function,
\begin{equation}\label{Sk}
\begin{aligned}
S_{k}(x):=&\|(x_k+\gamma S^{-1}Bx_{k-1})-(\gamma S^{-1}Bx+x)\|_{S}^2+2\langle u_k, x_k-x\rangle +L_{k-1}\|y_{k-1}-x_{k-1}\|_S^2 \\
&+\frac{\alpha+\gamma\mu(\varepsilon_5\mu+1)}{1-(1+\varepsilon_6)\gamma^2\mu^2}\|x_k-x_{k-1}\|_S^2,
\end{aligned}
\end{equation}
which helps to establish the weak convergence of the Algorithm \ref{algorithm4}.	

\begin{lemma}
\label{lemma4}
Suppose that Assumptions {\rm\ref{assumption1}} holds with $\gamma_k \equiv \gamma,$ $1-(1+\varepsilon_6)\gamma^2\mu^2 >0,$ for all $k \in \mathbb{N}$. Then for  $\{x_k\}_{k \in \mathbb{N}},$ $\{y_k\}_{k \in \mathbb{N}},$ and $\{u_k\}_{k \in \mathbb{N}}$ generated by Algorithm \ref{algorithm4} and  $\forall x^* \in \zer(A+B+C),$
\begin{equation}
S_{k+1}(x^*) \leq S_{k}(x^*)-(1- L_{k-1}-L_k-\gamma L_k^2\mu-\frac{\gamma\beta}{2}-\frac{\gamma}{\varepsilon_5}-\frac{(\alpha+\gamma\mu(\varepsilon_5\mu+1))(1+\frac{1}{\varepsilon_6})}{1-(1+\varepsilon_6)\gamma^2\mu^2})\|y_k-x_k\|_S^2 \label{lem4}
\end{equation}
holds for all $n \in \mathbb{N}_0.$
\end{lemma}

\begin{proof}
From Algorithm \ref{algorithm4}, we have
$$y_k=(M_k+A)^{-1}(M_kx_k-(B+C)x_k+\gamma^{-1}u_k),$$
which is equivalent to the inclusion
				\begin{equation}
					M_kx_k-M_ky_k-(B+C)x_k+\gamma^{-1}u_k \in Ay_k.     \label{51}
				\end{equation}

				Since $x^* \in \zer(A+B+C),$ it follows that
				\begin{equation}
					-(B+C)x^* \in Ax^*.                                 \label{52}
				\end{equation}
				Combining \eqref{51} and \eqref{52}, and using the monotonicity of $A$, we obtain				
                \begin{equation}
				0 \leq \langle M_kx_k-M_ky_k-(B+C)x_k+\gamma^{-1}u_k+(B+C)x^* , y_k-x^* \rangle. 					\label{53}
				\end{equation}
Multiplying both sides by $2\gamma$ and using the definition of $u_{k+1}$ yield
\begin{equation}
\aligned
0& \leq 2\langle Sx_k-Sy_k+u_k-u_{k+1}-\gamma(B+C)x_k+\gamma(B+C)x^* , y_k-x^* \rangle\\
& = 2\langle Sx_k-Sy_k, y_k-x^* \rangle +2\langle u_k-u_{k+1},y_k-x^* \rangle-2\gamma \langle Bx_k-Bx^*,y_k-x^*\rangle \\
& \quad -2\gamma\langle Cx_k-Cx^*,y_k-x^*\rangle. 	
\endaligned
\label{54}
\end{equation}

We now analyze each term on the right-hand side of \eqref{54}, separately. Using the definition of $x_{k+1}$, and the identity $2\langle x - y | y - z\rangle_S = \|x - z\|_S^2 - \|x - y\|_S^2 - \|y - z\|_S^2$, we obtain
				\begin{equation}
					\begin{aligned}
					   & 2\langle Sx_k -Sy_k, y_k-x^*\rangle \\
                       =&\|x_k-x^*\|_S^2-\|y_k-x^*\|_S^2-\|y_k-x_k\|_S^2 \\
                       =&\|x_k-x^*\|_S^2-\|x_{k+1}+\gamma S^{-1}Bx_k-\gamma S^{-1}Bx_{k-1}-x^*\|_S^2-\|y_k-x_k\|_S^2 \\
                       =&\|x_k-x^*\|_S^2-\|x_{k+1}-x^*\|_S^2-2\gamma\langle x_{k+1}-x^*,Bx_k-Bx_{k-1}\rangle \\
                       &-\gamma^2\|Bx_k-Bx_{k-1}\|_{S^{-1}}^2-\|y_k-x_k\|_S^2. \\
					\end{aligned}  \label{55}
				\end{equation}
Secondly, by using the Lipschitz property of $B$ and $u_{k}$, $\forall k \in \mathbb{N}$, we have
				\begin{equation}
					\begin{aligned}
					   & 2\langle u_k-u_{k+1}, y_k-x^*\rangle \\
                       =&2\langle u_k,x_k-x^*\rangle+2\langle u_k,y_k-x_k\rangle-2\langle u_{k+1},x_{k+1}-x^*\rangle \\
                       &-2\gamma\langle u_{k+1},S^{-1}Bx_k-S^{-1}Bx_{k-1}\rangle. \\
                       \leq &2\langle u_k,x_k-x^*\rangle+2\|u_k\|_{S^{-1}}\|y_k-x_k\|_S-2\langle u_{k+1},x_{k+1}-x^*\rangle \\
                       & +2\gamma\|u_{k+1}\|_{S^{-1}}\|Bx_k-Bx_{k-1}\|_{S^{-1}} \\
                       \leq&  2\langle u_k,x_k-x^*\rangle-2\langle u_{k+1},x_{k+1}-x^*\rangle+L_{k-1}\|y_{k-1}-x_{k-1}\|_S^2 \\
                       & +L_{k-1}\|y_{k}-x_{k}\|_S^2+\gamma L_k^2\mu\|y_k-x_k\|_S^2+\gamma\mu\|x_k-x_{k-1}\|_S^2. \\
					\end{aligned}  \label{56}
				\end{equation}
Moreover, by Lemma \ref{lem0}, we obtain
				\begin{equation}
					-2\gamma\langle Cx_k-Cx^* , y_k-x^* \rangle \leq \frac{\gamma\beta}{2}\|y_k-x_k\|_S^2. \label{57}
				\end{equation}
				Combining \eqref{54}--\eqref{57}, we get
				\begin{equation}
					\begin{aligned}
						& \|x_{k+1}-x^*\|_S^2+2\langle u_{k+1}, x_{k+1}-x^*\rangle\\
						\leq &\|x_k-x^*\|_S^2 +2\langle u_k,x_k-x^*\rangle+L_{k-1}\|y_{k-1}-x_{k-1}\|_S^2 \\
                        & -2\gamma \langle Bx_k-Bx^*,y_k-x^*\rangle-2\gamma\langle x_{k+1}-x^*,Bx_k-Bx_{k-1}\rangle-\gamma^2\|Bx_k-Bx_{k-1}\|_{S^{-1}}^2 \\
                        & + \gamma\mu\|x_k-x_{k-1}\|_S^2 -(1-L_{k-1}-\gamma L_k^2\mu-\frac{\gamma\beta}{2})\|y_k-x_k\|_S^2.
					\end{aligned} 					\label{58}
				\end{equation}
Finally, applying the monotonicity of $B$ and Young’s inequality, we derive
\begin{equation}
\begin{aligned}
&-2\gamma \langle Bx_k-Bx^*,y_k-x^*\rangle-2\gamma\langle x_{k+1}-x^*,Bx_k-Bx_{k-1}\rangle-\gamma^2\|Bx_k-Bx_{k-1}\|_{S^{-1}}^2 \\
= & -2\gamma \langle Bx_k-Bx^*,x_{k+1}-x^*\rangle -2\gamma \langle Bx_k-Bx^*,\gamma S^{-1}Bx_k-\gamma S^{-1}Bx_{k-1}\rangle \\
&+2\gamma\langle Bx_k-Bx_{k-1}, x^*-x_{k+1}\rangle-\gamma^2\|Bx_k-Bx_{k-1}\|_{S^{-1}}^2 \\
= & -2\gamma \langle Bx_k-Bx^*,x_{k+1}-x^*\rangle+\gamma^2(\|Bx_{k-1}-Bx^*\|_{S^{-1}}^2-\|Bx_k-Bx^*\|_{S^{-1}}^2 \\
& -\|Bx_k-Bx_{k-1}\|_{S^{-1}}^2)+2\gamma\langle Bx_k-Bx_{k-1},x_k-x_{k+1}\rangle+2\gamma\langle Bx_{k-1}-Bx^*,x_k-x^*\rangle\\
&+2\gamma\langle Bx^*-Bx_k,x_k-x^*\rangle-\gamma^2\|Bx_k-Bx_{k-1}\|_{S^{-1}}^2 \\
\leq & -2\gamma \langle Bx_k-Bx^*,x_{k+1}-x^*\rangle+2\gamma\langle Bx_{k-1}-Bx^*,x_k-x^*\rangle+\gamma^2\|Bx_{k-1}-Bx^*\|_{S^{-1}}^2 \\
&-\gamma^2\|Bx_k-Bx^*\|_{S^{-1}}^2+2\gamma\langle Bx_k-Bx_{k-1},x_k-x_{k+1}\rangle-2\gamma^2\|Bx_k-Bx_{k-1}\|_{S^{-1}}^2 \\
= & -2\gamma \langle Bx_k-Bx^*,x_{k+1}-x^*\rangle+2\gamma\langle Bx_{k-1}-Bx^*,x_k-x^*\rangle+\gamma^2\|Bx_{k-1}-Bx^*\|_{S^{-1}}^2 \\
&-\gamma^2\|Bx_k-Bx^*\|_{S^{-1}}^2+2\gamma\langle Bx_k-Bx_{k-1},x_k-y_k\rangle \\
\leq & -2\gamma \langle Bx_k-Bx^*,x_{k+1}-x^*\rangle+2\gamma\langle Bx_{k-1}-Bx^*,x_k-x^*\rangle+\gamma^2\|Bx_{k-1}-Bx^*\|_{S^{-1}}^2 \\
&-\gamma^2\|Bx_k-Bx^*\|_{S^{-1}}^2+\gamma\varepsilon_5\mu^2\|x_k-x_{k-1}\|_{S}^2+\frac{\gamma}{\varepsilon_5}\|y_k-x_k\|_S^2.\\
\end{aligned} 					\label{59}
\end{equation}
Substituting \eqref{59} into \eqref{58}, we obtain the following inequality:
\begin{equation*}
\begin{aligned}
& \|x_{k+1}-x^*\|_S^2+2\gamma \langle Bx_k-Bx^*,x_{k+1}-x^*\rangle+\gamma^2\|Bx_k-Bx^*\|_{S^{-1}}^2+2\langle u_{k+1}, x_{k+1}-x^*\rangle\\
\leq &\|x_k-x^*\|_S^2 +2\gamma\langle Bx_{k-1}-Bx^*,x_k-x^*\rangle+\gamma^2\|Bx_{k-1}-Bx^*\|_{S^{-1}}^2+2\langle u_k,x_k-x^*\rangle\\
& +L_{k-1}\|y_{k-1}-x_{k-1}\|_S^2+\gamma\mu(\varepsilon_5\mu+1)\|x_k-x_{k-1}\|_S^2\\
& -(1-L_{k-1}-\gamma L_k^2\mu-\frac{\gamma\beta}{2}-\frac{\gamma}{\varepsilon_5})\|y_k-x_k\|_S^2.
\end{aligned} 					\label{60}
\end{equation*}
Therefore,
\begin{equation}
\begin{aligned}
&\|(x_{k+1}+\gamma S^{-1}Bx_k)-(\gamma S^{-1}Bx^*+x^*)\|_{S}^2+2\langle u_{k+1}, x_{k+1}-x^*\rangle \\
&+\frac{\alpha+\gamma\mu(\varepsilon_5\mu+1)}{1-(1+\varepsilon_6)\gamma^2\mu^2}\|x_{k+1}-x_k\|_S^2\\
\leq &\|(x_k+\gamma S^{-1}Bx_{k-1})-(\gamma S^{-1}Bx^*+x^*)\|_{S}^2+2\langle u_k,x_k-x^*\rangle+L_{k-1}\|y_{k-1}-x_{k-1}\|_S^2\\
& +\gamma\mu(\varepsilon_5\mu+1)\|x_k-x_{k-1}\|_S^2-(1-L_{k-1}-\gamma L_k^2\mu-\frac{\gamma\beta}{2}-\frac{\gamma}{\varepsilon_5})\|y_k-x_k\|_S^2 \\
&+\frac{\alpha+\gamma\mu(\varepsilon_5\mu+1)}{1-(1+\varepsilon_6)\gamma^2\mu^2}\|x_{k+1}-x_k\|_S^2 \\
\leq &\|(x_k+\gamma S^{-1}Bx_{k-1})-(\gamma S^{-1}Bx^*+x^*)\|_{S}^2+2\langle u_k,x_k-x^*\rangle+L_{k-1}\|y_{k-1}-x_{k-1}\|_S^2\\
& +\gamma\mu(\varepsilon_5\mu+1)\|x_k-x_{k-1}\|_S^2-(1-L_{k-1}-\gamma L_k^2\mu-\frac{\gamma\beta}{2}-\frac{\gamma}{\varepsilon_5})\|y_k-x_k\|_S^2 \\
&+\frac{\alpha+\gamma\mu(\varepsilon_5\mu+1)}{1-(1+\varepsilon_6)\gamma^2\mu^2}((1+\varepsilon_6)\gamma^2\|Bx_k-Bx_{k-1}\|_{S^{-1}}^2+(1+\frac{1}{\varepsilon_6})\|y_k-x_k\|_S^2) \\
\leq &\|(x_k+\gamma S^{-1}Bx_{k-1})-(\gamma S^{-1}Bx^*+x^*)\|_{S}^2+2\langle u_k,x_k-x^*\rangle+L_{k-1}\|y_{k-1}-x_{k-1}\|_S^2\\
& +\frac{\alpha(1+\varepsilon_6)\gamma^2\mu^2+\gamma\mu(\varepsilon_5\mu+1)}{1-(1+\varepsilon_6)\gamma^2\mu^2}\|x_k-x_{k-1}\|_S^2\\
&-(1-L_{k-1}-\gamma L_k^2\mu-\frac{\gamma\beta}{2}-\frac{\gamma}{\varepsilon_5}-\frac{(\alpha+\gamma\mu(\varepsilon_5\mu+1))(1+\frac{1}{\varepsilon_6})}{1-(1+\varepsilon_6)\gamma^2\mu^2})\|y_k-x_k\|_S^2 \\
\leq &\|(x_k+\gamma S^{-1}Bx_{k-1})-(\gamma S^{-1}Bx^*+x^*)\|_{S}^2+2\langle u_k,x_k-x^*\rangle+L_{k-1}\|y_{k-1}-x_{k-1}\|_S^2\\
& +\frac{\alpha+\gamma\mu(\varepsilon_5\mu+1)}{1-(1+\varepsilon_6)\gamma^2\mu^2}\|x_k-x_{k-1}\|_S^2\\
&-(1-L_{k-1}-\gamma L_k^2\mu-\frac{\gamma\beta}{2}-\frac{\gamma}{\varepsilon_5}-\frac{(\alpha+\gamma\mu(\varepsilon_5\mu+1))(1+\frac{1}{\varepsilon_6})}{1-(1+\varepsilon_6)\gamma^2\mu^2})\|y_k-x_k\|_S^2
\end{aligned} 					\label{61}
\end{equation}
Then \eqref{lem4} is deduced by the definition of $S_k(x^*)$ in \eqref{Sk}.
\end{proof}

\begin{theorem}
\label{theorem4}
{
\noindent
Suppose that Assumptions {\rm\ref{assumption1}} holds with $\gamma_k \equiv \gamma$ for all $k \in \mathbb{N}$, let $\varepsilon_5, \ \varepsilon_6, \ \varepsilon_7>0$, $\alpha >0$ and there exists a constant $\epsilon >0$ such that the following conditions hold:
\vskip 1mm
\begin{itemize}
\item[{\rm(i)}] $1- L_{k-1}-L_k-\gamma L_k^2\mu-\frac{\gamma\beta}{2}-\frac{\gamma}{\varepsilon_5}-\frac{(\alpha+\gamma\mu(\varepsilon_5\mu+1))(1+\frac{1}{\varepsilon_6})}{1-(1+\varepsilon_6)\gamma^2\mu^2} \geq \epsilon,$	
\noindent
\item[{\rm(ii)}] $1-\varepsilon_7\mu(\varepsilon_5\mu+1)>0,$
\noindent
\item[{\rm(iii)}]$\max\{\varepsilon_7, \sqrt{\frac{1-\varepsilon_7\mu(\varepsilon_5\mu+1)}{(1+\varepsilon_6)\mu^2}}\} \leq \gamma \leq \sqrt{\frac{1}{(1+\varepsilon_6)\mu^2}}.$
\end{itemize}
Then the sequence $\{x_k\}_{k \in \mathbb{N}}$ generated by Algorithm \ref{algorithm4} converges weakly to a point $x^*$ in $\zer(A+B+C)$.
}
\end{theorem}

\begin{proof}
Let $x^{*}\in \zer(A+B+C)$. From the definition of $S_k(x^*)$ in \eqref{Sk}, the Lipschitz continuity of $\gamma M_k-S$, and the Cauchy-Schwarz inequality, there exists $\varepsilon_7>0$ such that
\begin{equation*}
\begin{aligned}
S_{k}(x^*)=&\|(x_k+\gamma S^{-1}Bx_{k-1})-(\gamma S^{-1}Bx^*+x^*)\|_{S}^2+2\langle u_k, x_k-x\rangle +L_{k-1}\|y_{k-1}-x_{k-1}\|_S^2 \\
&+\frac{\alpha+\gamma\mu(\varepsilon_5\mu+1)}{1-(1+\varepsilon_6)\gamma^2\mu^2}\|x_k-x_{k-1}\|_S^2 \\
\geq & \|x_k-x^*\|_{S}^2+2\gamma\langle x_k-x^*, Bx_{k-1}-Bx^* \rangle+\gamma^2\|Bx_{k-1}-Bx^*\|_{S^{-1}}^2 \\
& -L_{k-1}\|y_{k-1}-x_{k-1}\|_S^2-L_{k-1} \|x_k-x^*\|_{S}^2+L_{k-1}\|y_{k-1}-x_{k-1}\|_S^2 \\
&+\frac{\alpha+\gamma\mu(\varepsilon_5\mu+1)}{1-(1+\varepsilon_6)\gamma^2\mu^2}\|x_k-x_{k-1}\|_S^2 \\
=&(1-L_{k-1})\|x_k-x^*\|_{S}^2+2\gamma\langle (x_{k-1}-x^*)-(x_{k-1}-x_k), Bx_{k-1}-Bx^* \rangle \\
&+\gamma^2\|Bx_{k-1}-Bx^*\|_{S^{-1}}^2+\frac{\alpha+\gamma\mu(\varepsilon_5\mu+1)}{1-(1+\varepsilon_6)\gamma^2\mu^2}\|x_k-x_{k-1}\|_S^2 \\
\geq & (1-L_{k-1})\|x_k-x^*\|_{S}^2-\frac{\gamma}{\varepsilon_7}\|x_k-x_{k-1}\|_S^2 -\gamma\varepsilon_7\|Bx_{k-1}-Bx^*\|_{S^{-1}}^2 \\
&+\gamma^2\|Bx_{k-1}-Bx^*\|_{S^{-1}}^2+\frac{\gamma\mu(\varepsilon_5\mu+1)}{1-(1+\varepsilon_6)\gamma^2\mu^2}\|x_k-x_{k-1}\|_S^2 \\
= & (1-L_{k-1})\|x_k-x^*\|_{S}^2+(\frac{\gamma\mu(\varepsilon_5\mu+1)}{1-(1+\varepsilon_6)\gamma^2\mu^2}-\frac{\gamma}{\varepsilon_7})\|x_k-x_{k-1}\|_S^2 \\
&+ \gamma(\gamma-\varepsilon_7)\|Bx_{k-1}-Bx^*\|_{S^{-1}}^2. \\
\end{aligned}  \label{63}
\end{equation*}
Thanks to the conditions (ii), (iii), the sequence $\{S_{k}(x^*)\}_{k \in \mathbb{N}}$ is nonnegative, $\forall k \in \mathbb{N}.$  Since it is also nonincreasing, it is convergent. Let $N \in \mathbb{N}_0.$ Summing both sides of \eqref{lem4} over $k = 1,2,\cdots,N$ yields
\begin{equation*}
\begin{aligned}
&\sum_{k=1}^N(1- L_{k-1}-L_k-\gamma L_k^2\mu-\frac{\gamma\beta}{2}-\frac{\gamma}{\varepsilon_5}-\frac{(\alpha+\gamma\mu(\varepsilon_5\mu+1))(1+\frac{1}{\varepsilon_6})}{1-(1+\varepsilon_6)\gamma^2\mu^2})\|y_k-x_k\|_S^2 \\ &\leq S_{1}(x^*)- S_{N+1}(x^*)\\
& \leq S_{1}(x^*) < + \infty.
\end{aligned}  \label{64}
\end{equation*}
From condition (i), we have that $\lim\limits_{k\rightarrow \infty}\|y_k-x_k\|_S=0.$  Let $a_k=x_k+\gamma S^{-1}Bx_{k-1}$, and $a=x^*+\gamma S^{-1}Bx^*$. The following inequalities can be derived from \eqref{61}.
\begin{equation*}
\begin{aligned}
&\|a_{k+1}-a\|_S^2 +2\langle u_{k+1}, x_{k+1}-x^*\rangle +L_{k}\|y_{k}-x_{k}\|_S^2 \\
&+(\alpha+\frac{\alpha(1+\varepsilon_6)\gamma^2\mu^2+\gamma\mu(\varepsilon_5\mu+1)}{1-(1+\varepsilon_6)\gamma^2\mu^2})\|x_{k+1}-x_{k}\|_S^2, \\
\leq &\|a_{k}-a\|_S^2 +2\langle u_k, x_k-x^*\rangle +L_{k-1}\|y_{k-1}-x_{k-1}\|_S^2 \\
& +\frac{\alpha(1+\varepsilon_6)\gamma^2\mu^2+\gamma\mu(\varepsilon_5\mu+1)}{1-(1+\varepsilon_6)\gamma^2\mu^2}\|x_k-x_{k-1}\|_S^2 \\
& -(1- L_{k-1}-L_k-\gamma L_k^2\mu-\frac{\gamma\beta}{2}-\frac{\gamma}{\varepsilon_5}-\frac{(\alpha+\gamma\mu(\varepsilon_5\mu+1))(1+\frac{1}{\varepsilon_6})}{1-(1+\varepsilon_6)\gamma^2\mu^2})\|y_k-x_k\|_S^2 \\
\leq &\|a_{k}-a\|_S^2 +2\langle u_k, x_k-x^*\rangle +L_{k-1}\|y_{k-1}-x_{k-1}\|_S^2 \\
& +\frac{\alpha(1+\varepsilon_6)\gamma^2\mu^2+\gamma\mu(\varepsilon_5\mu+1)}{1-(1+\varepsilon_6)\gamma^2\mu^2}\|x_k-x_{k-1}\|_S^2 \\
& -(1- L_{k-1}-L_k-\gamma L_k^2\mu-\frac{\gamma\beta}{2}-\frac{\gamma}{\varepsilon_5}-\frac{(\alpha+\gamma\mu(\varepsilon_5\mu+1))(1+\frac{1}{\varepsilon_6})}{1-(1+\varepsilon_6)\gamma^2\mu^2})\|y_k-x_k\|_S^2.
 \end{aligned} \label{65}
\end{equation*}
The above inequalities imply that
\begin{equation*}
\begin{aligned}
\begin{split}
&\lim_{k\rightarrow \infty}S_k(x^*) \ \hbox{exists}, \\
&\lim_{k\rightarrow \infty}\|y_{k}-x_{k}\|_S^2=0, \\
& \lim_{k\rightarrow \infty}\|x_{k+1}-x_{k}\|_S^2=0, \\
\end{split}  \label{66}
\end{aligned}
\end{equation*}
and
\begin{equation*}
\lim\limits_{k\rightarrow \infty}S_k(x^*)=\lim\limits_{k\rightarrow \infty}\|a_k-a\|_S^2. \label{67}
\end{equation*}
On the other hand, the definition of $a_k$ yields that
\begin{equation*}
a_k = (I + \gamma S^{-1} B)x_k+\gamma S^{-1} Bx_{k-1}-\gamma S^{-1} Bx_k,
\end{equation*}
and
\begin{equation*}
x_k = J_{\gamma S^{-1} B}(a_k - \gamma(S^{-1} Bx_{k-1}- S^{-1} Bx_{k})).
\end{equation*}
Since  $\lim\limits_{k\rightarrow \infty}\|Bx_{k-1}-Bx_{k}\|_{S^{-1}}=0$ and $J_{\gamma S^{-1} B}$ is nonexpansive, it follows that the sequence $\{x_k\}_{k \in \mathbb{N}}$ is bounded. Consequently, the sequence $\{x_k\}_{k \in \mathbb{N}}$ has at least one weakly convergent subsequence. Without loss of generality, we may assume that $x_{k_n} \rightharpoonup \bar{x} \in \mathcal{H}$ as $n\rightarrow \infty$.
Define
\begin{equation*}
\Delta_k := M_kx_k-M_ky_k-(B+C)(x_k-y_k)+\gamma^{-1}u_k,
\end{equation*}
It follows from \eqref{51}  that $(y_k , \Delta_k) \in \gra(A+B+C)$, for all $k \in \mathbb{N}.$ Since $M_k$ is $2\gamma^{-1}$- Lipschitz continuous w.r.t. $S$, $B$ is $\mu$-Lipschitz continuous w.r.t. $S$, and $C$ is $\beta^{-1}$-cocoercive w.r.t. $S$, we obtain that
\begin{equation*}
\aligned \| \Delta_k \|_{S^{-1}}
& \leq \|M_kx_k-M_ky_k\|_{S^{-1}}+\|(B+C)(x_k-y_k)\|_{S^{-1}}+\frac{1}{\gamma}\|u_k\|_{S^{-1}}  \\
& \leq \frac{2}{\gamma}\|x_k-y_k\|_S+(\mu+\beta)\|x_k-y_k\|_S+\frac{L_{k-1}}{\gamma}\|x_{k-1}-y_{k-1}\|_S.
\endaligned
\end{equation*}
Based on the preceding arguments, it follows that $\Delta_{k_n}\rightarrow 0.$ The remaining part of the proof proceeds similarly to that of Theorem \ref{theorem2}, and is therefore omitted.
\end{proof}

\begin{remark}
	\rm We now present two special cases of Algorithm \ref{algorithm4}.
	\begin{itemize}
		\item[{\rm (1)}] When $S=\Id$, $\gamma_k \equiv \gamma$, and $M_k\equiv\frac{\Id}{\gamma}$, we have $L_k=0$ and $u_k=0.$  Hence, Algorithm \ref{algorithm4} reduces to the ORFBS algorithm \eqref{ORFBS}. Furthermore, the conditions in Theorem \ref{theorem4} become
		\begin{itemize}
			\item[{\rm(i)}] $1-\frac{\gamma\beta}{2}-\frac{\gamma}{\varepsilon_5}-\frac{(\alpha+\gamma\mu(\varepsilon_5\mu+1))(1+\frac{1}{\varepsilon_6})}{1-(1+\varepsilon_6)\gamma^2\mu^2} \geq \epsilon,$	
			\noindent
			\item[{\rm(ii)}] $1-\varepsilon_7\mu(\varepsilon_5\mu+1)>0,$
			\noindent
			\item[{\rm(iii)}]$\max\{\varepsilon_7, \sqrt{\frac{1-\varepsilon_7\mu(\varepsilon_5\mu+1)}{(1+\varepsilon_6)\mu^2}}\} \leq \gamma \leq \sqrt{\frac{1}{(1+\varepsilon_6)\mu^2}}.$
		\end{itemize}
		
		\item[{\rm (2)}] When $A=A_1+A_2,$ where $A_1:\mathcal{H} \rightarrow 2^{\mathcal{H}}$ is maximally monotone, $A_2:\mathcal{H} \rightarrow 2^{\mathcal{H}}$ is $L$-Lipschitz continuous, and $A_1+A_2$ is maximally monotone. Furthermore, let $S=\Id,$ $\gamma_k\equiv \gamma$, and $M_k\equiv M=\frac{\Id}{\gamma}-A_2$, $\forall k \in \mathbb{N}$. Then Algorithm \ref{algorithm4} becomes
		\begin{equation}
			\left\{
			\begin{array}{lr}
				y_k =J_{\gamma A_1}(x_k-\gamma (A_2+B+C)x_k-\gamma(A_2y_{k-1}-A_2x_{k-1})), & \\
				x_{k+1}=y_k-\gamma Bx_k+\gamma Bx_{k-1}.& \\
			\end{array}
			\right.			\label{new-ORFBS}
		\end{equation} 		
		Now, $\gamma M-\Id$ is $\gamma L$-Lipschitz continuous, the conditions in Theorem \ref{theorem4} become
\begin{itemize}
	\item[{\rm(i)}] $1-2\gamma L-\gamma^3 L^2\mu-\frac{\gamma\beta}{2}-\frac{\gamma}{\varepsilon_5}-\frac{(\alpha+\gamma\mu(\varepsilon_5\mu+1))(1+\frac{1}{\varepsilon_6})}{1-(1+\varepsilon_6)\gamma^2\mu^2} \geq \epsilon,$	
	\noindent
	\item[{\rm(ii)}] $1-\varepsilon_7\mu(\varepsilon_5\mu+1)>0,$
	\noindent
	\item[{\rm(iii)}]$\max\{\varepsilon_7, \sqrt{\frac{1-\varepsilon_7\mu(\varepsilon_5\mu+1)}{(1+\varepsilon_6)\mu^2}}\} \leq \gamma \leq \sqrt{\frac{1}{(1+\varepsilon_6)\mu^2}}.$
\end{itemize}
	\end{itemize}
\end{remark}

\subsubsection{Linear convergence}
In this subsection, we establish the $R$-linear convergence of the sequence $\{x_k\}_{k \in \mathbb{N}}$ generated by Algorithm \ref{algorithm4}, assuming that $A$ is strongly monotone.

\begin{theorem}
	\label{linear3}
	{
		\noindent
		Suppose that Assumption {\rm\ref{assumption1}} holds and that $A$ is $\rho$-strongly monotone.  Assume further that the conditions in Theorem \ref{theorem4} are satisfied and that there exist constants $0<\varepsilon_7 <1,$ $\varepsilon_8>0$  such that $0<\varepsilon_7 <\frac{1}{\gamma},$ and $\frac{\gamma}{A+\alpha}<\varepsilon_8<\gamma$,  and define
		\begin{equation*}
			\begin{aligned}
				t=\min\big\{ &\frac{\sqrt{(1+\gamma\mu(2+\gamma\mu)+L_k)^2+8\gamma^2\rho\mu(1+\gamma\mu)(1-\gamma\varepsilon_7)}-(1+\gamma\mu(2+\gamma\mu)+L_k)}{2\gamma\mu(1+\gamma\mu)}, \\
				&  \frac{\nu}{2L_k}, \frac{\alpha-2\gamma^2\rho\mu^2(\frac{1}{\varepsilon_7}-\gamma)}{A+2\gamma^2\rho\mu^2(\frac{1}{\varepsilon_7}-\gamma)}, \ \frac{\varepsilon_8(A+\alpha)-\gamma}{\gamma}\big\},
			\end{aligned}
		\end{equation*}
		where $\nu=1- L_{k-1}-L_k-\gamma L_k^2\mu-\frac{\gamma\beta}{2}-\frac{\gamma}{\varepsilon_5}-\frac{(\alpha+\gamma\mu(\varepsilon_5\mu+1))(1+\frac{1}{\varepsilon_6})}{1-(1+\varepsilon_6)\gamma^2\mu^2},${\tiny } $A=\frac{\alpha(1+\varepsilon_6)\gamma^2\mu^2+\gamma\mu(\varepsilon_5\mu+1)}{1-(1+\varepsilon_6)\gamma^2\mu^2}$,  with $\alpha > \max\{2\gamma^2\rho\mu^2(\frac{1}{\varepsilon_7}-\gamma), 1-(1+\varepsilon_6)\gamma^2\mu^2-\gamma\mu(\varepsilon_5\mu+1)\}.$
		Then the sequence $\{x_k\}_{k \in \mathbb{N}}$ generated by Algorithm \ref{algorithm4} converges $R$-linearly to a point $x^*$ in $\zer(A+B+C)$.
	}
\end{theorem}

\begin{proof}
By the strong monotonicity of $A$, inequality \eqref{53} can be reformulated as
	\begin{equation*}
		\rho \|y_k-x^*\|_S^2 \leq \langle M_kx_k-M_ky_k-(B+C)x_k+\gamma^{-1}u_k+(B+C)x^* , y_k-x^* \rangle. 					\label{100}
	\end{equation*}
Furthermore, the following inequality, which follows from the proof of Lemma \ref{lemma4}, holds:
\begin{equation}
\begin{aligned}
& \|x_{k+1}-x^*\|_S^2+2\gamma \langle Bx_k-Bx^*,x_{k+1}-x^*\rangle+\gamma^2\|Bx_k-Bx^*\|_{S^{-1}}^2+2\langle u_{k+1}, x_{k+1}-x^*\rangle\\
& +(A+\alpha)\|x_{k+1}-x_k\|_S^2+L_k\|y_k-x_k\|_S^2+2\gamma\rho\|y_k-x^*\|_S^2 \\
\leq &\|x_k-x^*\|_S^2 +2\gamma\langle Bx_{k-1}-Bx^*,x_k-x^*\rangle+\gamma^2\|Bx_{k-1}-Bx^*\|_{S^{-1}}^2+2\langle u_k,x_k-x^*\rangle\\
& +A\|x_k-x_{k-1}\|_S^2+L_{k-1}\|y_{k-1}-x_{k-1}\|_S^2-\nu\|y_k-x_k\|_S^2,
\end{aligned} 					\label{101}
\end{equation}
where $A=\frac{\alpha(1+\varepsilon_6)\gamma^2\mu^2+\gamma\mu(\varepsilon_5\mu+1)}{1-(1+\varepsilon_6)\gamma^2\mu^2}$. By Young's inequality, we obtain
\begin{equation}
	\label{102}
	\begin{aligned}
		&2\gamma\rho\|y_k-x^*\|_S^2 \\
		= & 2\gamma\rho\|x_{k+1}+\gamma S^{-1}Bx_k-\gamma S^{-1}Bx_{k-1}-x^*\|_S^2 \\
		= & 2\gamma\rho(\|x_{k+1}-x^*\|_S^2+2\gamma\langle x_{k+1}-x^*,Bx_k-Bx_{k-1}\rangle+\gamma^2\|Bx_k-Bx_{k-1}\|_{S^{-1}}^2) \\
		\geq &  2\gamma\rho(1-\gamma\varepsilon_7)\|x_{k+1}-x^*\|_S^2+2\gamma^2\rho\mu^2(\gamma-\frac{1}{\varepsilon_7})\|x_k-x_{k-1}\|_{S}^2.
	\end{aligned} 					
\end{equation}
For any $x^*$ in $\zer(A+B+C)$, define
$$d_k(x^*)=2\gamma\langle Bx_{k-1}-Bx^*,x_k-x^*\rangle+\gamma^2\|Bx_{k-1}-Bx^*\|_{S^{-1}}^2+2\langle u_k,x_k-x^*\rangle+L_{k-1}\|y_{k-1}-x_{k-1}\|_S^2, $$
$$e_k=\nu\|y_k-x_k\|_S^2.$$
Combining  \eqref{101} and \eqref{102}, we obtain the following inequality
\begin{equation}
\label{103}
\begin{aligned}
&(1+2\gamma\rho(1-\gamma\varepsilon_7))\|x_{k+1}-x^*\|_S^2+d_{k+1}+(A+\alpha)\|x_{k+1}-x_k\|_{S}^2+e_k \\
\leq & \|x_{k}-x^*\|_S^2+d_k+(A+2\gamma^2\rho\mu^2(\frac{1}{\varepsilon_7}-\gamma))\|x_k-x_{k-1}\|_{S}^2.
\end{aligned} 	
\end{equation}
Using the Lipschitz continuity of $u_{k+1}$ and $B$,  we can bound $d_{k+1}(x^*)$ as follows:
\begin{equation}
\label{104}
\begin{aligned}
d_{k+1}(x^*)=&2\gamma\langle Bx_{k}-Bx^*,x_{k+1}-x^*\rangle+\gamma^2\|Bx_{k}-Bx^*\|_{S^{-1}}^2+2\langle u_{k+1},x_{k+1}-x^*\rangle \\
&+L_k\|y_k-x_k\|_S^2 \\
\leq &\gamma\mu\|x_k-x^*\|_S^2+\gamma\mu\|x_{k+1}-x^*\|_S^2+\gamma^2\mu^2\|x_k-x^*\|_S^2+L_k\|y_k-x_k\|_S^2 \\
&+L_k\|x_{k+1}-x^*\|_S^2+L_k\|y_k-x_k\|_S^2\\
= &(\gamma\mu(1+\gamma\mu))\|x_k-x^*\|_S^2+(\gamma\mu+L_k)\|x_{k+1}-x^*\|_S^2+2L_k\|y_k-x_k\|_S^2.
		\end{aligned}
	\end{equation}

	By the choice of $t$, we have
\begin{equation}
	\label{105}
	\begin{aligned}
		td_{k+1}(x^*) \leq&t\gamma\mu(1+\gamma\mu)\|x_k-x^*\|_S^2+e_k \\
		& +(2\gamma\rho(1-\gamma\varepsilon_7)-t\gamma\mu(1+\gamma\mu)-t-t^2\gamma\mu(1+\gamma\mu))\|x_{k+1}-x^*\|_S^2.
\end{aligned}
\end{equation}
Furthermore, we obtain
\begin{equation}
\label{106}
(A+2\gamma^2\rho\mu^2(\frac{1}{\varepsilon_7}-\gamma))\|x_{k}-x_{k-1}\|_{S}^2\leq\frac{A+\alpha}{1+t}\|x_{k}-x_{k-1}\|_{S}^2.
\end{equation}
	Combining \eqref{103}--\eqref{106}, we obtain the following inequality:
	\begin{equation*}
		\label{107}
		\begin{aligned}
		&(1+t)(1+t\gamma\mu(1+\gamma\mu))\|x_{k+1}-x^*\|_S+d_{k+1}+\frac{A+\alpha}{1+t}\|x_{k+1}-x_k\|_{S}^2) \\
		\leq &(1+t\gamma\mu(1+\gamma\mu))\|x_k-x^*\|_S^2 +d_k+\frac{A+\alpha}{1+t}\|x_{k}-x_{k-1}\|_{S}^2.
	\end{aligned}
	\end{equation*}
On the other hand, using the Lipschitz continuity of $u_{k+1}$, the monotonicity of $B$, and Young's inequality, we can deduce that there exists a constant $k_3>0$ such that
\begin{equation*}
\begin{aligned}
&(1+t\gamma\mu(1+\gamma\mu))\|x_{k+1}-x^*\|_S+d_{k+1}+\frac{A+\alpha}{1+t}\|x_{k+1}-x_k\|_{S}^2 \\
\geq &(1+t\gamma\mu(1+\gamma\mu))\|x_{k+1}-x^*\|_S+2\gamma\langle Bx_{k}-Bx^*,(x_k-x^*)-(x_k-x_{k+1})\rangle \\
&+\gamma^2\|Bx_{k}-Bx^*\|_{S^{-1}}^2-L_k\|x_{k+1}-x^*\|_S^2+\frac{A+\alpha}{1+t}\|x_{k+1}-x_k\|_{S}^2\\
\geq &(1+t\gamma\mu(1+\gamma\mu)-L_k)\|x_{k+1}-x^*\|_S-\frac{\gamma}{\varepsilon_8}\|x_k-x_{k+1}\|_S^2-\gamma\varepsilon_8\|Bx_k-Bx^*\|_{S^{-1}}^2 \\
&+\gamma^2\|Bx_{k}-Bx^*\|_{S^{-1}}^2+\frac{A+\alpha}{1+t}\|x_{k+1}-x_k\|_{S}^2\\
= &(1+t\gamma\mu(1+\gamma\mu)-L_k)\|x_{k+1}-x^*\|_S+\gamma(\gamma-\varepsilon_8)\|Bx_k-Bx^*\|_{S^{-1}}^2 \\
&+(\frac{A+\alpha}{1+t}-\frac{\gamma}{\varepsilon_8})\|x_{k+1}-x_k\|_{S}^2\\
			\geq & k_3\|x_{k+1}-x^*\|_S^2.
		\end{aligned} 	
	\end{equation*}
	Consequently, we have
	\begin{equation*}
		\begin{aligned}
			\label{107}
			k_3\|x_{k+1}-x^*\|^2_S & \leq(1+t\gamma\mu(1+\gamma\mu))\|x_{k+1}-x^*\|_S+d_{k+1}+\frac{A+\alpha}{1+t}\|x_{k+1}-x_k\|_{S}^2\\
			&\leq \frac{(1+t\gamma\mu(1+\gamma\mu))\|x_k-x^*\|_S^2 +d_k+\frac{A+\alpha}{1+t}\|x_{k}-x_{k-1}\|_{S}^2}{1+t} \\
			&\leq \ldots \leq \frac{(1+t\gamma\mu(1+\gamma\mu))\|x_1-x^*\|_S^2 +d_1+\frac{A+\alpha}{1+t}\|x_{1}-x_{0}\|_{S}^2}{(1+t)^k}.
		\end{aligned}
	\end{equation*}
Hence, the sequence $\{x_k\}_{k \in \mathbb{N}}$ generated by Algorithm \ref{algorithm4} converges $R$-linearly to a point $x^*$ in $\zer(A+B+C)$.
\end{proof}

\section{Numerical experiments}	
In this section, we focus on a special case of Algorithm \ref{algorithm4}, namely, Algorithm \eqref{new-ORFBS}. We assess the performance of Algorithm \eqref{new-ORFBS} with the ORFBS algorithm \eqref{ORFBS}. Note that, for the numerical experiments in this section, Algorithm \eqref{remark2+} (derived in Remark \ref{remark2})  is equivalent to the SRFBS algorithm \eqref{SRFBS}, and therefore omitted from the comparative analysis. All implementations were written in MATLAB and executed on a computer equipped with an Intel Core i7-11800H CPU (2.30GHz) and 32GB RAM, running Windows 11.

\subsection{Nonlinear constrained optimization problems}	
{\rm
In this subsection, we conduct numerical experiments on constrained nonlinear optimization problems to validate the effectiveness of the proposed algorithm. Specifically, we consider the following constrained optimization problem:
	\begin{equation}
		\label{fh}
		\min_{x \in C}f(x)+h(x),
	\end{equation}
	where $C= \{ x \in \mathcal{H} | ( \forall i \in \{1,...,q\}) \ g_i(x) \leq 0\}.$ Here, $f: \mathcal{H} \rightarrow ( -\infty ,+\infty]$ is a proper, convex  and  lower semi-continuous function, for each $i \in \{1,...,q\}$, $g_i : \dom(g_i) \subset \mathcal{H} \rightarrow \mathbb{R}$ and $h: \mathcal{H} \rightarrow \mathbb{R}$ are $C^1$ convex functions in $\nt \dom g_i$ and $\mathcal{H}$, respectively, with $\nabla h$ being $\beta$-Lipschitz continuous. The solution to problem \eqref{fh} corresponds to  a saddle point of the Lagrangian function
	\begin{equation*}
		L(x,u)=f(x)+h(x)+u^{\top}g(x)-\iota_{\mathbb{R}_+^q}(u),
	\end{equation*}
	where $\iota_{\mathbb{R}_+^q}$ denotes the indicator function of $\mathbb{R}_+^q$. Under standard regularity conditions, this solution can be obtained by solving the following monotone inclusion problem  \cite{RT,FBHF}: find $x \in Y$ such that there exists $u \in \mathbb{R}_{+}^q$ satisfying
	\begin{equation}
		\label{ABCxu}
		(0,0) \in A(x,u)+B(x,u)+C(x,u),
	\end{equation}
	where $Y \subset \mathcal{H}$ is a nonempty closed convex set modeling prior information on the solution. The operators are defined as follows: $A:(x,u)\mapsto\partial f(x)\times{N_{\mathbb{R}_+^q}u}$ is maximally monotone, $B:(x,u)\mapsto(\sum_{i=1}^q u_i \nabla g_i(x),-g_1(x),...,-g_q(x))$ is nonlinear, monotone and continuous, and $C:(x,u)\mapsto(\nabla h(x),0)$ is $\frac{1}{\beta}$-cocoercive.
	
	Let  $\mathcal{H}=\mathbb{R}^N$, $f=\iota_{[0,1]^N}$, $g_i(x)=d_i^{\top}x$ for all $i \in  \{1,...,q\}$ with $d_1,\ldots,d_q\in \mathbb{R}^N$, and  $h(x)=\frac{1}{2}\|Gx-b\|^2$ where $G$ is an $m \times N$ real matrix, $N=2m$, $b\in \mathbb{R}^m$. The operators in \eqref{ABCxu} then take the form:
	\begin{equation*}
		\label{A1}
		\aligned
		&A:(x,u)\mapsto\partial{\iota_{[0,1]^N}(x)}\times{N_{\mathbb{R}_+^p}u},\\
		& B:(x,u)\mapsto(D^\top u,-Dx),\\
		& C:(x,u)\mapsto(G^\top(Gx-b),0),\\
		\endaligned
	\end{equation*}
	where  $x\in \mathbb{R}^N$, $u\in \mathbb{R}_+^q$ and $D=[d_1,\ldots,d_q]^\top$. Here, $A$ is maximally monotone, $B$ is $L$-Lipschitz continuous with $L=\|D\|$, and $C$ is  $\frac{1}{\beta}$-cocoercive with $\beta =\|G\|^{2}$. Therefore, problem \eqref{ABCxu} can be solved by ORFBS algorithm \eqref{ORFBS}.
	
	Alternatively, we can decompose $B$ as $B_1+B_2$, where
	\begin{equation}
		\label{B1B2}
		\aligned
		& B_1:(x,u)\mapsto(\frac{1}{2}D^\top u,-\frac{1}{2}Dx),\\
		& B_2:(x,u)\mapsto(\frac{1}{2}D^\top u,-\frac{1}{2}Dx).\\
		\endaligned
	\end{equation}
It is straightforward to verify that $A+B_1$ is maximally monotone, while $B_1$ and $B_2$ are both $\frac{L}{2}$-Lipschitz continuous, with $L=\|D\|$. Consequently, problem \eqref{ABCxu} can also be solved by  Algorithm \eqref{new-ORFBS}.
	
In the numerical experiments, the matrices $G,D$, the vector $b$, and the initial values $(x_{0},u_{0})$  are all randomly generated. The following stopping criterion is employed:
	$$
	E_k=\frac{\|(x_{k+1}-x_k,u_{k+1}-u_k)\|}{\|(x_k,u_k)\|}<10^{-6},
	$$
We test eight different problem sizes,  with ten randomly generated instances for each size. Table \ref{results-1} reports the average number of iterations (denoted as "av.iter") and the average CPU time (denoted as "av.time")  over these ten instances. As shown in Table \ref{results-1}, Algorithm \eqref{new-ORFBS} outperforms the ORFBS algorithm in both iteration count and computational efficiency in most cases.
	\begin{table}[h!]
		\centering
		\footnotesize
		\renewcommand\arraystretch{1.5}
		\setlength\tabcolsep{5pt}
		\caption{Computational results with ORFBS and Algorithm \eqref{new-ORFBS}}
		\begin{tabular}{c c c c c c}
			\hline
			\multicolumn{2}{c}{\textbf{Problem size}} &
			\multicolumn{2}{c}{\textbf{av.iter}} &
			\multicolumn{2}{c}{\textbf{av.time (s)}} \\
			\cline{1-2} \cline{3-4} \cline{5-6}
			$N$ & $q$ & ORFBS & Algorithm \eqref{new-ORFBS} & ORFBS & Algorithm \eqref{new-ORFBS} \\
			\hline
			\multirow{4}{*}{2000}
			& 100  & 7267.8    & \textbf{6806.1}    & 25.669  & \textbf{15.831} \\
			& 200  & 5950.1  & \textbf{5742.5}  & 35.587    & \textbf{17.538} \\
			& 500  & 4441.3    & \textbf{4244.7}  & 77.558   & \textbf{31.192} \\
			& 1000 & 3187.4 & \textbf{3134}  & 169.68    & \textbf{58.53} \\
			\hline
			\multirow{3}{*}{4000}
			& 100  & 8724.7   & \textbf{8525.4}    & 282.33   & \textbf{212.37} \\
			& 200  & 6445.4  & \textbf{6401.8}  & 278.65    & \textbf{186.07} \\
			& 500  & 3427.7 & \textbf{3713.6}  & 368.03    & \textbf{158.26} \\
			& 1000  & \textbf{811.7} & 1295  & 157.74   & \textbf{79.902} \\
			\hline
		\end{tabular} \label{results-1}
	\end{table}

\subsection{Portfolio optimization problem}	
This subsection considers the classical mean-variance portfolio optimization problem for risk minimization. Suppose a portfolio consists of $n$ distinct assets, where each asset $i$ has a random rate of return with expected value $m_i$.  The objective is to determine investment weights $x_i$ that minimize portfolio risk while ensuring a specified minimum expected return $r$.  Let $H \in \mathbb{R}^{225 \times 225}$ denote the covariance matrix of asset returns ($\|H\| = 0.2263$). The optimization problem can be formulated as:
\begin{equation}\label{portfolio-problem}
	\begin{aligned}
		\min_{x\in R^{225}} & \frac{1}{2}x^{T}Hx, \\
		s.t. & \sum_{i=1}^{75} x_i \geq 0.3, \
		\sum_{i=76}^{150} x_i \geq 0.3, \
		\sum_{i=151}^{225} x_i \geq 0.3, \\
		& \sum_{i=1}^{225}m_i x_i \geq r, \ \sum_{i=1}^{225}x_i = 1, 0\leq x_i \leq 1, i = 1, \cdots, 225.
	\end{aligned}
\end{equation}
Define the constraint mapping $g(x) = Dx + b$ where:
\begin{equation*}
	\left[
	\begin{array}{ccccccccc}
		-m_1 & \cdots & -m_{75} & -m_{76} & \cdots & -m_{150} & -m_{151} & \cdots & -m_{225} \\
		-1 & \cdots & -1 & 0 & \cdots & 0 & 0 & \cdots & 0\\
		0 & \cdots & 0 & -1 & \cdots & -1 & 0 & \cdots & 0\\
		0 & \cdots & 0 & 0 & \cdots & 0 & -1 & \cdots & -1
	\end{array}
	\right]
	\left[
	\begin{array}{cccc}
		x_1\\
		x_2\\
		\vdots\\
		x_{225}
	\end{array}
	\right]
	+\left[
	\begin{array}{cccc}
		r\\
		0.3\\
		0.3\\
		0.3
	\end{array}
	\right]
\end{equation*}
and the feasible set $C_2 = \{x\in \mathbb{R}^n | \sum_{i=1}^{n}x_i = 1,\ 0\leq x_i \leq 1, i = 1, \cdots, n\}$. Solutions correspond to saddle points of the Lagrangian:
\begin{equation*}
	L(x,u)=\frac{1}{2}x^{T}Hx+\iota_{C_2}(x)+u^{\top}g(x)-\iota_{\mathbb{R}_+^4}(u).
\end{equation*}
Under standard qualification conditions, the solution to problem \eqref{portfolio-problem} can be obtained  by solving the following monotone inclusion problem: find $x \in \mathbb{R}^{225}$ such that $\exists u \in \mathbb{R}_{+}^4$,
\begin{equation*}
	(0,0) \in A(x,u)+B(x,u)+C(x,u),
\end{equation*}
where the operators are defined as
\begin{equation*}
	\label{A2}
	\aligned
	&A:(x,u)\mapsto\partial{\iota_{C_2}(x)}\times{N_{\mathbb{R}_+^4}u},\\
	&B:(x,u)\mapsto(D^\top u,-Dx-b),\\
	&C:(x,u)\mapsto(Hx,0).\\
	\endaligned
\end{equation*}
Here, $A$ is maximally monotone, $B$ is $L$-Lipschitz continuous with $L =\|D\|$, and $C$ is $\beta^{-1}$-cocoercive with $\beta =\|H\|$. Therefore, the ORFBS algorithm \eqref{ORFBS} can be employed to solve problem \eqref{portfolio-problem}. Moreover, following the decomposition in \eqref{B1B2}, we express $B$ as the sum $B_1 + B_2$. Consequently, Algorithm \eqref{new-ORFBS} provides an alternative method for solving problem \eqref{portfolio-problem}.

This study employs the portfolio optimization formulation \eqref{portfolio-problem}, which is adopted from the MATLAB optimization case library. The benchmark dataset used in the experiment is sourced from the OR-Library \cite{ChangCOR2000}, where the asset returns $m_i$ range from $-0.008489$ to $0.003971$. We evaluate the algorithms under three target return levels, $r = 0.001$, $0.002$, and $0.003$, in our experimental evaluations. As presented in Table \ref{results-2}, the solution accuracy achieved by the ORFBS algorithm is comparable to that of Algorithm \eqref{new-ORFBS}. Moreover, Figure \ref{solution-portfolio} illustrates that the resulting portfolio weight distributions exhibit consistent sparsity patterns across all target levels.

\begin{table}[htbp]
	\small
	\centering
	\caption{The objective function values (Obj), number of iterations (Iter), and CPU time (in seconds) for the compared methods on the portfolio optimization problem.}
	\begin{tabular}{c|c|ccc}
		\hline
		$r$ &  Methods &  Obj   &  Iter  & CPU  \\
		\hline
		\hline
		\multirow{2}[1]{*}{0.001}
		& ORFBS   & $1.6386e-4$        & $\textbf{248954}$       & $\textbf{6.1836}$       \\
		& Algorithm \eqref{new-ORFBS}   & $1.6386e-4$      & $249453$       & $6.5429$       \\
		\hline
		\multirow{2}[1]{*}{0.002}
		& ORFBS   & $2.0097e-4$        & $\textbf{271736}$       & $\textbf{6.6391}$       \\
		& Algorithm \eqref{new-ORFBS}   & $2.0097e-4$        & $257843$       & $6.8260$       \\
		\hline
		\multirow{2}[1]{*}{0.003}
		& ORFBS   & $2.7686e-4$        & $217194$       & $\textbf{5.7700}$       \\
		& Algorithm \eqref{new-ORFBS}   & $2.7686e-4$        & $\textbf{212883}$       & $6.5494$       \\
		\hline
	\end{tabular}\label{results-2}
\end{table}
\begin{figure}[h]
	\centering
	\setlength{\abovecaptionskip}{-3pt}
	\setlength{\belowcaptionskip}{-2pt} 
	\captionsetup[subfigure]{labelformat=simple, skip=2pt} 
	
	\subfigure[$r=0.001$]{
		\scalebox{0.3}{\includegraphics{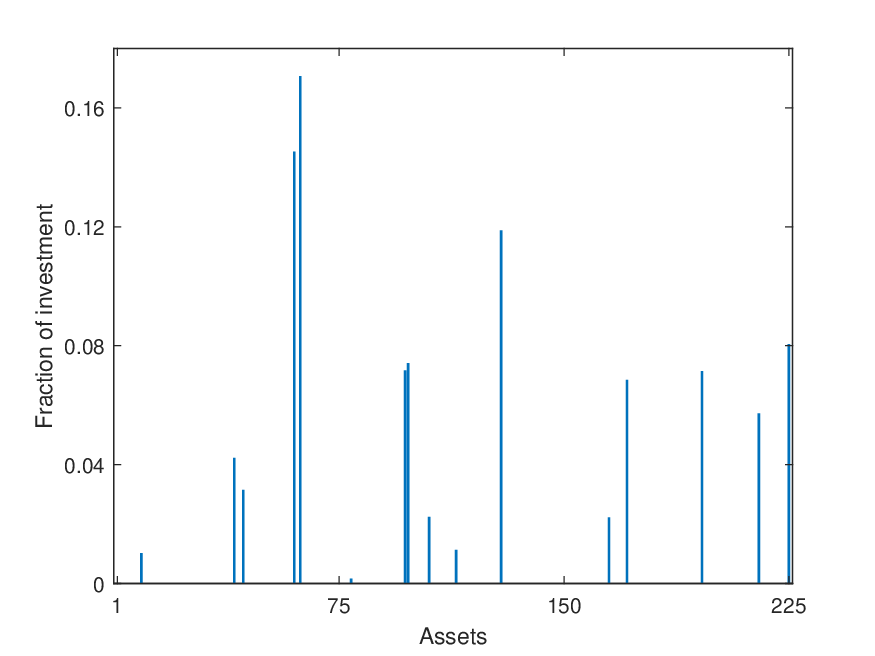}}
	} \hspace{1mm} 
	\subfigure[$r=0.002$]{
		\scalebox{0.3}{\includegraphics{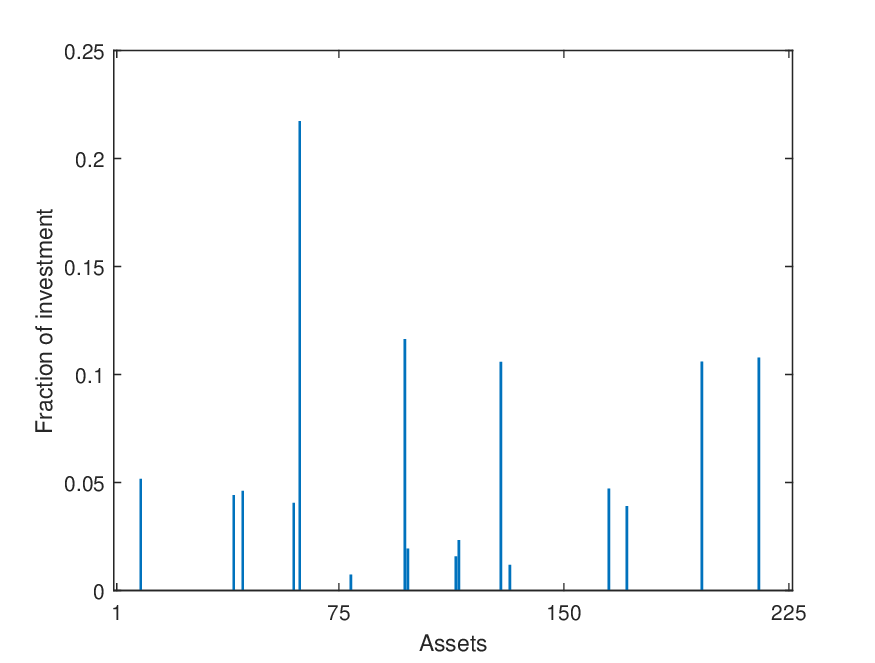}}
	} \hspace{1mm} 
	\subfigure[$r=0.003$]{
		\scalebox{0.3}{\includegraphics{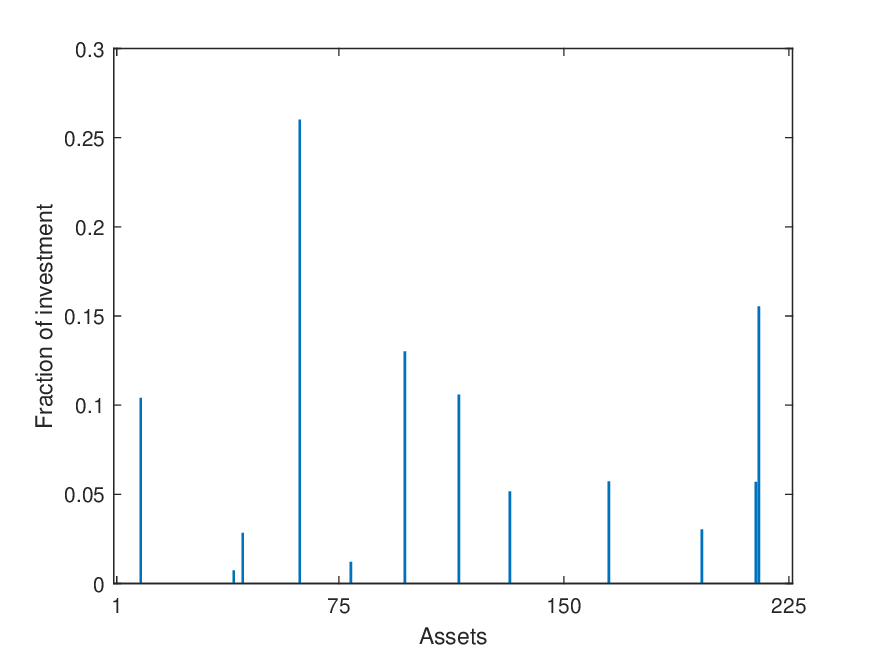}}
	}
	
	\vspace{3mm} 
	
	\subfigure[$r=0.001$]{
		\scalebox{0.3}{\includegraphics{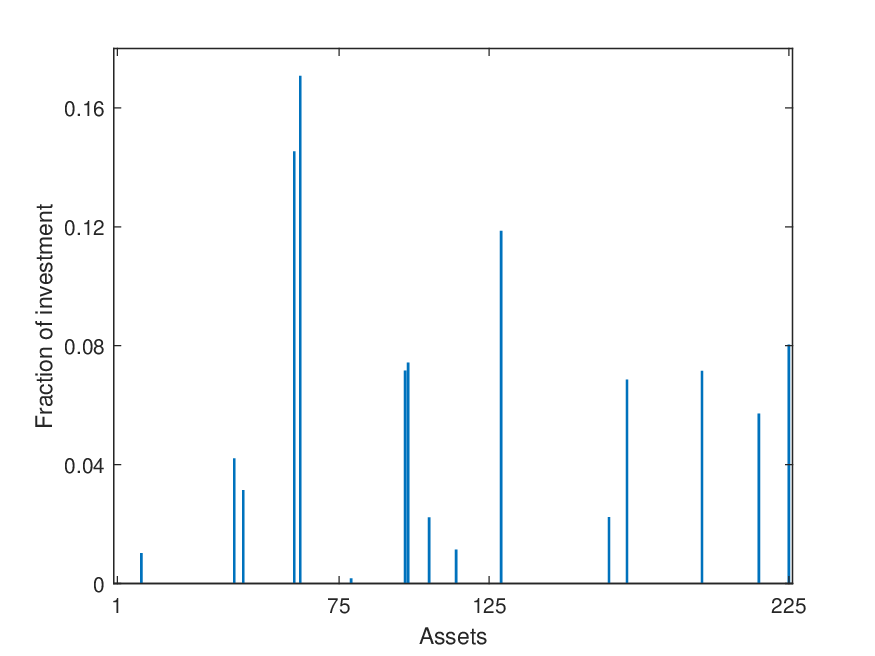}}
	} \hspace{1mm} 
	\subfigure[$r=0.002$]{
		\scalebox{0.3}{\includegraphics{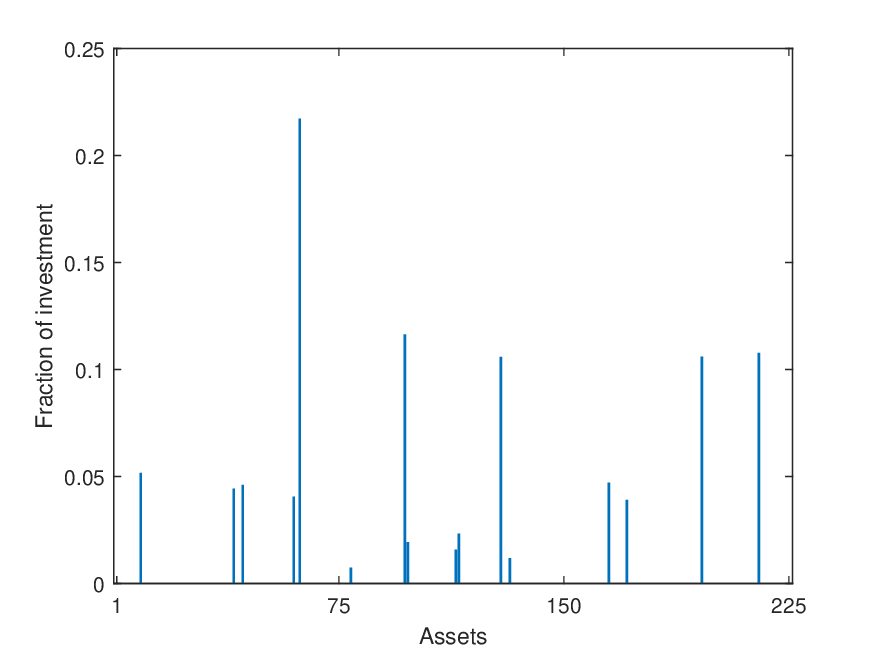}}
	} \hspace{1mm} 
	\subfigure[$r=0.003$]{
		\scalebox{0.3}{\includegraphics{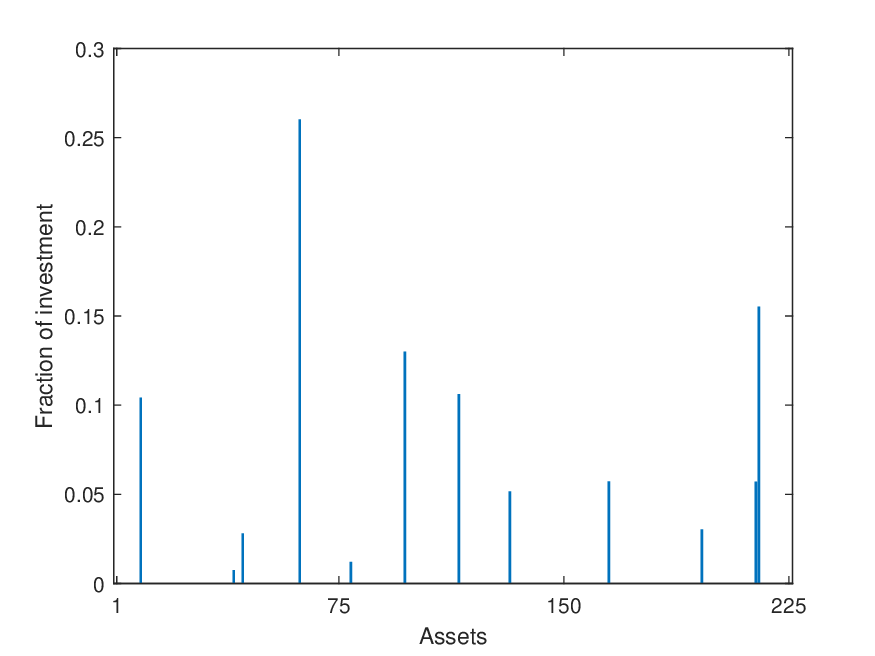}}
	}
	\vspace{3mm}
	\caption{Portfolio optimization solutions for problem \eqref{portfolio-problem}: ORFBS algorithm \eqref{ORFBS} (top row)  versus Algorithm \eqref{new-ORFBS} (bottom row).}
	\label{solution-portfolio}
\end{figure}


\section{Conclusions}

This study investigated structured monotone inclusion problems that consist of the sum of three operators: a maximally monotone operator, a maximally monotone and Lipschitz continuous operator, and a cocoercive operator. To address these problems, we developed three novel splitting algorithms by enhancing the semi-forward-reflected-backward, semi-reflected-forward-backward, and outer reflected forward-backward splitting frameworks with a nonlinear momentum term. Under suitable step-size conditions, we rigorously established the weak convergence of the proposed algorithms. Furthermore, when the strong monotonicity assumption holds, we proved their linear convergence rates. Comprehensive numerical experiments were conducted on both synthetic datasets and real-world quadratic programming problems in portfolio optimization. The results consistently verified the effectiveness and superior performance of the proposed algorithms in terms of convergence speed and solution accuracy compared with existing methods.

\section*{Funding}

This work was supported  by the National Natural Science Foundations of China (12031003, 12571491, 12571558), the Guangzhou Education Scientific Research Project 2024 (202315829), the Innovation Research Grant NO.JCCX2025018 for Postgraduate of Guangzhou University, and the Jiangxi Provincial Natural Science Foundation (20224ACB211004).

\section*{Competing Interests}

The authors declare no competing interests.


\end{document}